\documentclass[11pt]{amsart}
\usepackage{amsmath,amssymb,mathtools}
\usepackage[margin=1.15in]{geometry}
\usepackage[hidelinks]{hyperref}
\usepackage{iftex}
\ifPDFTeX
  \input{glyphtounicode}
\fi
\newcommand{\CC}{\mathbb C}
\newcommand{\HH}{\mathbb H}
\newcommand{\PP}{\mathbb P}
\newcommand{\ZZ}{\mathbb Z}
\newcommand{\vth}{\vartheta}
\newcommand{\ii}{\mathrm i}

\newcommand{\Span}{\operatorname{span}}
\newcommand{\rel}{\operatorname{rel}}
\makeatletter
\@namedef{subjclassname@2020}{%
  \textup{2020} Mathematics Subject Classification}
\makeatother
\newtheorem{theorem}{Theorem}
\newtheorem{proposition}[theorem]{Proposition}
\newtheorem{lemma}[theorem]{Lemma}
\newtheorem{corollary}[theorem]{Corollary}
\numberwithin{theorem}{section}
\numberwithin{equation}{section}

\theoremstyle{remark}
\newtheorem{remark}[theorem]{Remark}
\begin{document}
\title{The four-dimensional Sklyanin algebras revisited}
\author{S. Paul Smith}
\address{Department of Mathematics, Box 354350, University of Washington,
Seattle, WA 98195}
\email{smith@math.washington.edu}
\date{}
\begin{abstract}
We give a theta-function and projective-geometric account of the
four-dimensional Sklyanin algebras.  Starting with the space of order-four
theta functions whose zero divisors sum to zero on an elliptic curve, we
define a family of algebras $A(\eta\,\vert\,\tau)$ by means of the
translation-by-$\eta$ graph on
$E_\tau=\CC/(\ZZ\oplus\ZZ\tau)$ and the vertices
$e_0,e_1,e_2,e_3$ of the four singular quadrics that contain the associated
elliptic normal quartic $E_\tau$.  We use Jacobi's addition formulas to
produce six quadratic relations for $A(\eta\,\vert\,\tau)$ that coincide with the six relations that Sklyanin
obtained from Baxter's elliptic solution to the quantum Yang--Baxter
equation when $\eta\notin\frac12(\ZZ+\ZZ\tau)$.  Thus, the family of four-dimensional Sklyanin algebras coincides
with the algebras $A(\eta\,\vert\,\tau)$.  The Yang--Baxter equation plays no role in our definition of $A(\eta)$.  We
use the geometric definition of $A(\eta)$ to show that the Sklyanin algebra is
isomorphic to Odesskii and Feigin's algebra
$Q_{4,1}(-\tfrac12\eta)$.  The proof avoids the calculations in previous
proofs of this result.  Following Rains--Ruijsenaars, we describe a two-dimensional
space $C_\eta$ of $2$-tensors (theta functions in two variables) whose
non-zero elements vanish on the translation graph but not at all four
vertex pairs $(e_i,e_i)$; the image
of $C_\eta$ in $A(\eta\,\vert\,\tau)_2$ is a two-dimensional space of central
elements.  We show that the description of $C_\eta$, the centrality
of its image, and several other results follow from
the geometric definition of $A(\eta)$.
When $\overline\eta\notin E_\tau[4]$, an important example of this is
an explicit description, in terms of theta functions in two variables,
of the central elements $\Omega(\zeta)$ in
$A(\eta\,\vert\,\tau)_2$ that annihilate the various families of line modules;
the existence of such elements, without explicit formulas, was proved by
Levasseur--Smith~\cite{LS}. The relations of the form $u_1 \otimes v_2-u_2 \otimes v_1$ that were found in \cite{LS} (associated to secant lines and line modules)
are essentially the same as the theta functions $F_\eta(a,b,c_1,c_2;z_1,z_2)$ defined by Rains--Ruijsenaars.
\end{abstract}
\maketitle
\tableofcontents
\enlargethispage{4pt}

\section{Introduction}

Throughout this paper, $\HH$ denotes the upper-half complex plane and
$(\eta,\tau)$ denotes a point in $\CC\times\HH$.
We will use the notation
\[
 \Lambda_\tau=\ZZ+\ZZ\tau
 \qquad\text{and}\qquad
 E_\tau=\CC/\Lambda_\tau,
 \qquad
 \overline\eta=\eta+\Lambda_\tau\in E_\tau.
\]

For $n\geq1$, we write
$E_\tau[n]=\{\xi\in E_\tau\mid n\xi=0\}$.

Several different-looking algebras depending on $(\eta,\tau)$ have been
called four-dimensional Sklyanin algebras.  There are Sklyanin's
algebras $S(\eta\,\vert\,\tau)$ in
\cite{Sklyanin1982,Sklyanin1983}, which are defined when
$\overline{\eta}\notin E_\tau[2]$, and the algebras
$Q_{4,1}(\eta\,\vert\,\tau)$, defined by Feigin and Odesskii in
\cite{FeiginOdesskiiSklyanin}.\footnote{%
Feigin and Odesskii denoted their algebras by
$Q_{n,k}(E_\tau,\overline{\eta})$, but Chirvasitu--Kanda--Smith denoted
them by $Q_{n,k}(\eta\,\vert\,\tau)$ in \cite{CKSModular} and proved that
$Q_{n,k}(\eta\,\vert\,\tau)$ depends, up to isomorphism, only on the pair
$(E_\tau,\overline{\eta})$, thus justifying Feigin and Odesskii's
notation.  Thus $Q_{4,1}(E_\tau,\overline{\eta})$ and the
present $Q_{4,1}(\eta\,\vert\,\tau)$ denote the same algebra.}
The algebras $S(\eta\,\vert\,\tau)$ belong to a slightly larger class of
algebras denoted $S(\alpha,\beta,\gamma)$ by Smith--Stafford
\cite{SmithStafford}.  Stafford later introduced several related
algebras in \cite{StaffordRegularity}; their defining relations are
examined geometrically in \cite{StaffordRelations}.  These families are usually defined by
generators and relations.  In Appendix~B of \cite{RR},
Rains--Ruijsenaars define
$\mathcal R(\eta\,\vert\,\tau)=TV/(R_\eta^{\mathrm{RR}})$ for every
$\eta\in\CC$, where $V$ is the space of order-four theta functions whose zeros sum to 0 in $E_\tau$ and $R_\eta^{\mathrm{RR}}$ is a particular subspace of
$V^{\otimes2}$ spanned by certain two-variable theta functions.  They prove that
$\mathcal R(\eta\,\vert\,\tau)\cong S(\eta\,\vert\,\tau)$ when
$\overline{\eta}\notin E_\tau[2]$; their appendix was an important
stimulus for this paper.

One purpose of this paper is to bring more order to this situation by
introducing a geometrically defined algebra $A(\eta\,\vert\,\tau)$.  The
geometric definition allows us to prove some fundamental properties of
the Sklyanin algebras using algebraic-geometric and theta-function
arguments rather than calculations that manipulate generators and
relations.  Although it is not a panacea, it offers a new viewpoint, and
permits some new proofs and some new results.

Since $\tau$ will be fixed throughout the paper, we will often
abbreviate the names of these algebras by omitting the $\tau$ in the
notation $(\eta\,\vert\,\tau)$.

The ``new'' geometrically defined algebra
$A(\eta\,\vert\,\tau)=A(\eta)$ is defined for all
$(\eta,\tau)\in\CC\times\HH$.  We will show that
\begin{equation}\label{eq.three.main.isomorphisms}
 \begin{aligned}
 A(\eta)&\cong S(\eta)
   &&\text{if }\overline{\eta}\notin E_\tau[2],\\
 A(\eta)&= \mathcal R(\eta)
   &&\text{for all }\eta\in\CC,\\
 A(\eta)&\cong Q_{4,1}\left(-\tfrac12\eta\right)
   \cong Q_{4,1}\left(\tfrac12\eta\right)^{\mathrm{op}}
   &&\text{for all }\eta\in\CC.
 \end{aligned}
\end{equation}
It is defined as follows.  Embed $E_\tau$ in a projective space
$\PP V^*\cong\PP^3$ as an elliptic normal quartic via the complete
linear system $|\mathcal O_{E_\tau}(4[0])|$.  Let
$e_0,e_1,e_2,e_3$ be the vertices of the four singular
quadrics in $\PP V^*$ that contain $E_\tau$, and define
$R_\eta\subset V^{\otimes2}$ to be the subspace of elements  that
vanish at $(z,z+\eta)$ for all $z\in E_\tau$ and at all four $(e_i,e_i)$.
Define
\[
 A(\eta):=\frac{TV}{(R_\eta)},
\]
where, throughout this paper, $TV$ and $SV$ denote the tensor algebra
and the symmetric algebra on $V$, respectively, and $(R_\eta)$ is the
ideal generated by $R_\eta$.  That a four-dimensional Sklyanin algebra can
be defined in this way was first stated by Levasseur--Smith
\cite[Theorem~1.2]{LS}.  The roots of this observation lie in the earlier
paper by Smith and Stafford \cite{SmithStafford}.

We use Mumford's notation $\vartheta_{ab}(z\mid\tau)$ or, more
succinctly, $\vartheta_{ab}(z)$, with $a,b\in\{0,1\}$, for the theta
functions with half-integer characteristics and modulus $\tau$
\cite[p.~17]{MumfordTataI}.  We also use the numbered theta function notation of
Kharchev--Zabrodin, which is related to
Mumford's by
\begin{equation}\label{eq.numbered.theta.conventions}
 \theta_1=-\vartheta_{11},\qquad \theta_2=\vartheta_{10},\qquad
 \theta_3=\vartheta_{00},\qquad \theta_4=\vartheta_{01}.
\end{equation}
We also define
\begin{equation}\label{eq.theta.definition}
 \theta(z):=\theta_1(z\mid\tau)=-\vartheta_{11}(z\mid\tau);
\end{equation}
this is an odd function whose zeros are the points of
$\Lambda_\tau$.  The definitions of $\theta_1,\ldots,\theta_4$ and
$\theta$ agree with those in Appendix~B of
Rains--Ruijsenaars~\cite{RR} and with the notation in
\cite{Sklyanin1983}.  Appendix~\ref{app:theta.dictionary} gives a translation
dictionary among theta-function notations and those used
in this paper.

Sklyanin defined his eponymous algebras via generators and relations
that he derived from Baxter's elliptic solution to the $2\times2$
quantum Yang--Baxter equation
\cite{Sklyanin1982,Sklyanin1983}.
When $\overline\eta\notin E_\tau[2]$, we denote his algebra by
$S(\eta)$.  It is generated by elements $S_0,S_1,S_2,S_3$ subject to
the six relations
\begin{equation}\label{eq.Sklyanin.original.relations}
 \begin{cases}
 [S_0,S_\alpha]=\ii J_{\beta\gamma}\{S_\beta,S_\gamma\},\\
 [S_\alpha,S_\beta]=\ii\{S_0,S_\gamma\},
 \end{cases}
\end{equation}
where $(\alpha,\beta,\gamma)$ ranges over the cyclic permutations of
$(1,2,3)$, and
\[
 J_{12}=-\alpha_{01},\qquad
 J_{23}=-\alpha_{10},\qquad
 J_{31}=-\alpha_{00},
\]
where the $\alpha_{ab}$ are the elliptic functions of $\eta$ with
modulus $\tau$ defined in \eqref{eq.alpha.ab}, and, for
elements $x,y$ in a tensor algebra,
\[
 [x,y]=x\otimes y-y\otimes x,
 \qquad
 \{x,y\}=x\otimes y+y\otimes x.
\]
Sklyanin notes that the $\alpha_{ab}$ satisfy the identity
\begin{equation}\label{eq.Sklyanin.parameter.identity}
 \alpha_{00}+\alpha_{01}+\alpha_{10}
   +\alpha_{00}\alpha_{01}\alpha_{10}=0.
\end{equation}

For this reason Smith--Stafford \cite{SmithStafford} study a family of
algebras $S(\alpha_{00},\alpha_{01},\alpha_{10})$, defined over an
arbitrary algebraically closed field $\Bbbk$ of characteristic
not~$2$, where
$\alpha_{00},\alpha_{01},\alpha_{10}$ are arbitrary scalars satisfying
\eqref{eq.Sklyanin.parameter.identity}.
Here we relabel their generators $x_0,x_1,x_2,x_3$ as
$x_{11},x_{00},x_{01},x_{10}$, respectively;
$S(\alpha_{00},\alpha_{01},\alpha_{10})$ is the free algebra
$\Bbbk\langle x_{11},x_{00},x_{01},x_{10}\rangle$ modulo the six
relations
\begin{equation}\label{eq.Sklyanin.relations}
 \begin{cases}
 [x_{11},x_{ab}]-\alpha_{ab}\{x_{ij},x_{kl}\}=0,\\
 \{x_{11},x_{ab}\}-[x_{ij},x_{kl}]=0,
 \end{cases}
\end{equation}
as $(ab,ij,kl)$ ranges over the cyclic permutations of
$(00,01,10)$.  (In this paper we always work over the base field
$\CC$.)
Sklyanin's $S(\eta)$ equals
$S(\alpha_{00},\alpha_{01},\alpha_{10})$ with
\begin{equation}\label{eq.Sklyanin.generators.to.x}
 S_0=x_{11},\qquad S_1=-\ii x_{10},\qquad
 S_2=\ii x_{00},\qquad S_3=-\ii x_{01}.
\end{equation}
The main results in \cite{SmithStafford} are proved under the assumption
that
$\{\alpha_{00},\alpha_{01},\alpha_{10}\}\cap\{0,\pm1\}=\varnothing$;
Corollary~\ref{cor.alpha.parameters.and.four.torsion} will show that this
is equivalent to the assumption that
$\overline\eta\notin E_\tau[4]$.

Smith--Stafford \cite[p.~270]{SmithStafford} remark that
\eqref{eq.Sklyanin.parameter.identity} follows from Jacobi's identity,
\begin{equation}\label{eq.Jacobi.quartic.identity.intro}
 \vartheta_{00}(z)^4+\vartheta_{11}(z)^4
 =\vartheta_{01}(z)^4+\vartheta_{10}(z)^4.
\end{equation}
For \eqref{eq.Jacobi.quartic.identity.intro}, see Weber
\cite[vol.~III, p.~74, (13) and (14)]{Weber} and Jacobi's
theta-function lectures \cite[pp.~505--507, especially formula
(A)]{JacobiTheta}.  Jacobi's formula is the more general five-term
quartic identity; \eqref{eq.Jacobi.quartic.identity.intro} is one of
its specializations (see \cite[Chapter~I, \S5]{MumfordTataI} and
\cite{Koornwinder}).

Theorem~\ref{prop:addition-formulas-relations} shows that
$\dim R_\eta=6$ for every $\eta$.  We will use
Jacobi's addition formulas to give an explicit basis for $R_\eta$ that
depends holomorphically on $\eta$.  When
$\overline\eta\notin E_\tau[2]$, these basis elements are, after a
suitable change of generators, exactly Sklyanin's six
relations.  It is only on this locus that $S(\eta)$ is defined,
and there $A(\eta)\cong S(\eta)$.

Although the definition of $R_\eta$ requires its elements to vanish at
all four vertex pairs, Corollary~\ref{cor.two.vertices.determine.Reta}
shows that, when $\overline\eta\notin E_\tau[4]$, the translation graph
together with any two distinct vertex pairs already determines
$R_\eta$.  Vanishing at the other two vertex pairs is then automatic.

The purpose of this paper is to show that this geometric definition of
$A(\eta)$ is effective: structural properties of the four-dimensional Sklyanin
algebras can be derived from the graph, the four vertices, and theta
identities, without using calculations involving the generators and
relations in \eqref{eq.Sklyanin.relations} and those for $Q_{4,1}$ in
\S3.4 (see Remark~3.20).  For example, we will show that
\[
 A(\eta)\cong Q_{4,1}\left(-\tfrac12\eta\right),
\]
where $Q_{4,1}$ denotes the four-dimensional elliptic algebra defined
by Feigin and Odesskii~\cite{FeiginOdesskiiSklyanin}.  This isomorphism
has been asserted and sometimes proved by several people using
calculations involving the generators and relations in
\eqref{eq.Sklyanin.relations}.
We will prove the isomorphism by showing that the explicit relations
for $Q_{4,1}(-\tfrac12\eta)$ vanish on
$\{(z,z+\eta)\}$ and at the points $(e_i,e_i)$, so are contained in
$R_\eta$.  Since the relations for $Q_{4,1}$ span a six-dimensional
vector space, the isomorphism follows.

Section~3 proves the isomorphisms in \eqref{eq.three.main.isomorphisms}.
Each isomorphism is proved in a separate subsection and, for the
most part, those subsections can be read independently of the others.
There is one other subsection, namely \S3.1, which concerns a
two-dimensional subspace $C_\eta\subseteq V^{\otimes2}$  with the property that $C_\eta\cap R_\eta=\{0\}$. It is important for two reasons:  the algebra with relations $C_\eta\oplus R_\eta=K_\eta$ is the twisted
homogeneous coordinate ring for $E_\tau$ that is a quotient of $A(\eta)$; its image in $A(\eta)$ consists of degree-two central elements. 

Starting with the translation-by-$\{\frac{1}{4},\frac{1}{4}\tau\}$-action on theta functions, there is a canonical action of the Heisenberg group $H_4$ of order $4^3$ on $V$ that induces well-known 
$H_4$-actions on $S(\eta)$, $A(\eta)$, $R_\eta$, $K_\eta$, and $C_\eta$. Proposition~\ref{prop.H4.unique.complement} shows that $C_\eta$ is
canonical in the sense that it is the only complement to $R_\eta$ in
$K_\eta$ that is an $H_4$-submodule of $K_\eta$. 

{\emergencystretch=1em
Section~4 gives explicit descriptions of the point and line modules for
$A(\eta)$ that are parametrized by $E_\tau$ and the secant lines to
$E_\tau$.\footnote{A line $L$ in $\PP V^*$ is called a secant if $L\cap E_\tau$ is a scheme of length two. Thus, tangent lines are secant lines.}  The action of an element $f\in V$ on such a module 
is given by a formula involving the theta function $f(z)$.  Such formulas are
prominent in the work of Feigin--Odesskii, and we adapt their formulas
to the present situation.\par}

Section~5 concerns the degree-two central elements in $A(\eta)$.  We
show that the image of $C_\eta$ in $A(\eta)$ belongs to the center of
$A(\eta)$ for every $\eta$, including a direct argument at the
nonzero four-torsion parameters not covered by
\cite[Prop.~6.1]{CS}; we also give explicit representatives for $C_\eta$.
When $\overline\eta\notin E_\tau[4]$, Levasseur--Smith show that, for
each $\zeta\in\CC$, there is a
nonzero central element $\Omega(\overline\zeta)$ in $A(\eta)_2$ that
annihilates the line modules $M(p,q)$ for which
$p+q=\overline\zeta$, but they do not give a formula for it.  In
\S5.4, under the same hypothesis, we define an element (theta function)
$\Omega_\zeta(z_1,z_2)\in C_\eta$, denote its image
in $A(\eta)_2$ by $\Omega(\zeta)$, and show that it annihilates all
those line modules.  They also prove that
\[
 \CC\Omega(\overline\zeta)
 =\CC\Omega(-\overline\zeta-2\overline\eta).
\]
Our normalization is such that
$\Omega(\zeta)=\Omega(-\zeta-2\eta)$.  The proof of
Theorem~\ref{thm.explicit.Omega.z} uses the Rains--Ruijsenaars
description of $R_\eta^{\mathrm{RR}}$.  

Section~6 obtains the automorphisms in
\cite{SmithStaniszkis} by geometric rather than algebraic considerations.

Appendix~B of Rains--Ruijsenaars~\cite{RR} is an important point of
departure for this paper.  They define an algebra $\mathcal R(\eta)$ for every
$\eta$, introduce the two-dimensional space $C_\eta$, use Sklyanin's
analytic-difference-operator representations to prove that their
$C_\eta$ meets their relation space trivially, and prove that
$\mathcal R(\eta)\cong S(\eta)$ when
$\overline\eta\notin E_\tau[2]$.  Their formulas describe the entire
six-dimensional relation space and remain meaningful when $\overline{\eta} \in E_\tau[2]$.  We use their theta-function conventions and their
$C_\eta$, and their formulas later provide explicit relations and
central tensors.  Our approach complements theirs: we replace the difference-operator argument
they use to prove $C_\eta\cap R_\eta=\{0\}$ by four elementary vertex evaluations; for every $\eta$, their relation space is our $R_\eta$; 
our proof of this equality produces a basis of six relations  for $A(\eta)$ that are associated to some special secant and tangent lines.  Thus
our results overlap with theirs away from $E_\tau[2]$, and they also 
identify the limiting algebras when $\overline{\eta} \in E_\tau[2]$.

This paper reinforces the implicit, and sometimes explicit, message in
\cite{RR,CKS1} that the relation space, rather than individually normalized
relations or structure constants, is the natural object that varies with
the parameters.  Thus we regard $R_\eta$ as a point of
$\operatorname{Gr}(6,V^{\otimes2})$.  When
$\overline\eta\in E_\tau[2]$, Sklyanin's algebra is not defined, but the
varying six-plane $R_\eta$ has a well-defined limit at these points.
Chirvasitu--Kanda--Smith use the same principle for all $Q_{n,k}$, extending
the relation spaces across parameters at which individual displayed
relations are not defined.

\subsection{Acknowledgment of computational assistance}

This paper was developed through an extended collaboration between the
author and OpenAI's ChatGPT 5.6 Sol, using its Codex environment.  The system
assisted in reorganizing drafts, checking notation and cross-references,
locating sources, suggesting formulations and arguments, and compiling
and inspecting the typeset manuscript.  Anthropic's Claude Sonnet 5.5
was also used to write referee reports at various stages of the project.
The author reviewed and revised
the resulting text and is responsible for all mathematical statements
and proofs.

\section{Preliminaries}

\subsection{Theta-function conventions and the spaces
\texorpdfstring{$V_c$}{V-c}}

Our theta function notation was explained above in the paragraph that
contains \eqref{eq.numbered.theta.conventions}.  A summary of its
relation to the notation used by Mumford, Weber, Sklyanin,
Smith--Stafford, Kharchev--Zabrodin, and Rains--Ruijsenaars appears in
Appendix~\ref{app:theta.dictionary}.

We use the abbreviation
\[
\theta(u\pm v)=\theta(u+v)\theta(u-v).
\]
Define $V$ to be the space of order-$4$
theta functions whose zeros, regarded as points of $E_\tau$ and counted
with multiplicity, sum to $0$ in $E_\tau$. Thus $\dim V=4$. 
Since a theta function is determined up to a scalar multiple by the location of its zeros,
\begin{equation}\label{eq.V.zeros}
V=\left\{\lambda\prod_{j=1}^4\theta(z+a_j) \ \middle|\
 \lambda,a_1,a_2,a_3,a_4\in\CC,\ \sum_{j=1}^4a_j=0\right\}.
\end{equation}
Our $V$ equals the $V$ in Rains--Ruijsenaars~\cite[(B.1)]{RR}.
The space $V$ has basis
\[
\{\vartheta_{00}(2z),\vartheta_{01}(2z),
  \vartheta_{10}(2z),\vartheta_{11}(2z)\}.
\]

\subsubsection{The notation \texorpdfstring{$Y_{ab}$, $X_{ab}$, and
$x_{ab}$}{Yab, Xab, and xab}.}

We define
$(\gamma_{00},\gamma_{01},\gamma_{10},\gamma_{11}):=(1,i,i,1)$, and
\begin{equation}\label{eq.definition.Xhat}
 Y_{ab}(z):=\gamma_{ab}\vartheta_{ab}(2z).
\end{equation}
For every $\eta\in\CC$,
$\{Y_{00},Y_{01},Y_{10},Y_{11}\}$ is a basis for $V$.
When $\overline\eta\notin E_\tau[2]$, we define
\begin{equation}\label{eq.definition.Xab}
 X_{ab}:=\vartheta_{ab}(\eta)Y_{ab}.
\end{equation}
Each $X_{ab}$ is nonzero, so
$\{X_{00},X_{01},X_{10},X_{11}\}$ is a basis for $V$ when
$\overline{\eta}\notin E_\tau[2]$.
Let $TV\to A(\eta)$ be the natural quotient map.  We use $X_{ab}$ for
the elements of $V\subset TV$ defined above and $x_{ab}$ for their
images in $A(\eta)$.  Under the isomorphism $A(\eta)\cong S(\eta)$
proved in Theorem~\ref{prop:addition-formulas-relations}, these images
correspond to the standard generators $x_{ab}$ of
$S(\eta)=S(\alpha_{00},\alpha_{01},\alpha_{10})$ introduced in the
introduction.  Thus uppercase symbols will be used for elements of
$TV$, and lowercase symbols for their images in $A(\eta)$ or, under
this isomorphism, in $S(\eta)$. $\lozenge$

Let $W$ be the space of theta functions of order $2$
whose two zeros, counted with multiplicity, sum to $0$ in $E_\tau$.
Thus $\dim W=2$, and
\[
W=\{c\,\theta(v\pm\gamma)\mid c,\gamma\in\CC\}.
\]
When we wish to emphasize the argument $v$, we write $W(v)$ for this
space.

The  holomorphic map $\CC \to \PP V^*$ given  by 
\[
z\longmapsto
(\vartheta_{00}(2z),\vartheta_{01}(2z),
 \vartheta_{10}(2z),\vartheta_{11}(2z))
\]
factors through $E_\tau$. The induced map $E_\tau \to \PP V^*$ is a morphism of algebraic varieties whose image is isomorphic to $E_\tau$. 
This is the embedding associated to the complete linear system $|\mathcal O_{E_\tau}(4[0])|$. 
We identify $E_\tau$ with its image in $\PP V^*\cong\PP^3$. 

For $i,j\in\{0,1\}$, let $\epsilon_{ij}\in V^*$ be dual to
$\vartheta_{ij}(2z)$, and put $e_{ij}=[\epsilon_{ij}]\in\PP V^*$.
Corollary~\ref{cor.eij.are.vertices} shows that these are the vertices
of the singular quadrics containing $E_\tau$.

Sometimes, when we do not need to be specific, we will write
$e_0,e_1,e_2,e_3$ for the vertices.

When the $Y_{ab}$ are used as homogeneous coordinates on $\PP V^*$,
the points $e_{00},e_{01},e_{10},e_{11}$ are precisely the points at which
three of the four   $Y_{ab}$ vanish.

We identify $V^{\otimes2}$ with the space of holomorphic functions
$f(z_1,z_2)$ such that
\begin{enumerate}
\item for each fixed $z_1$, the function $z\mapsto f(z_1,z)$ belongs to $V$;
\item for each fixed $z_2$, the function $z\mapsto f(z,z_2)$ belongs to $V$.
\end{enumerate}

We now introduce the translated theta spaces and the basis used in
\cite{CKS1}.

For each $c\in\CC$, we write $V_c$ for the space of order-four theta
functions whose zeros, counted with multiplicity, in a period
parallelogram sum to $c$ in $E_\tau$.  Hence $V_c$ is the space denoted
by $\Theta_{4,c}(\Lambda_\tau)$ in \cite{CKS1}, which we will denote by
$\Theta_{4,c}(\tau)$ in this paper.  Equivalently,
\begin{align}
 V_c
 &=\left\{\lambda\prod_{j=1}^4\theta(z+a_j) \ \middle|\
 \lambda,a_1,a_2,a_3,a_4\in\CC,\
 \sum_{j=1}^4a_j=c\text{ in }E_\tau\right\}
 \label{eq.Vc.zeros}\\
 &=H^0(E_\tau,\mathcal L_c),
 \qquad
 \mathcal L_c=\mathcal O_{E_\tau}\!\left(4\left[\tfrac14c\right]\right).
 \notag
\end{align}
Clearly, $V_{c+1}=V_c$ and $V_0$ is the space $V$ that we defined
above.  Following the notation in \cite[p.~179]{CKS1}, almost, we
define
\[
 \Theta_4(\tau)=\Theta_{4,3/2}(\tau)=V_{3/2}=V_{1/2}.
\]

Let $e(z)=\exp(2\pi\ii z)$.  

The order-one theta function used in
\cite[Section~2.1]{CKS1} is, in terms of the function in
\eqref{eq.theta.definition},
\[
 \theta^{\mathrm{CKS}}(z)
   =-\ii e\!\left(\tfrac12z-\tfrac18\tau\right)\theta(z).
\]
Specializing \cite[Prop.~2.6]{CKS1} to
$n=4$, define
\[
 \theta_\alpha(z)
 =e\!\left(
       \alpha z+\tfrac{\alpha}{8}
       +\tfrac{\alpha(\alpha-4)}{8}\tau
     \right)
   \prod_{m=0}^{3}
   \theta^{\mathrm{CKS}}\!\left(
       z+\tfrac m4+\tfrac{\alpha}{4}\tau
   \right),
 \qquad \alpha\in\ZZ.
\]
Then $\theta_{\alpha+4}=\theta_\alpha$ and
$\{\theta_0,\theta_1,\theta_2,\theta_3\}$ is the basis for
$\Theta_4(\tau)$ that is used in \cite{CKS1}.  These functions are
\emph{not} the same as those defined in the introduction.  This basis is
characterized up to scalar multiples by 
\begin{equation}
\label{eq.CKS.shift.quarter}
\theta_\alpha\!\left(z+\tfrac14\right)
   =\ii^\alpha\theta_\alpha(z).
\end{equation}

\begin{lemma}
\label{prop.translation.V.Vprime}
Let $d\in\CC$. The map
\[
 \Psi_d:V\longrightarrow V_{4d},
 \qquad (\Psi_df)(z)=f(z-d),
\]
is a vector-space isomorphism.  In particular, 
\[
 \Psi_{1/8}:V\xrightarrow{\sim}\Theta_4(\tau),
 \qquad (\Psi_{1/8}f)(z)=f\!\left(z-\tfrac18\right).
\]
\end{lemma}

\begin{proof}
If the zeros of $f$ in a period parallelogram sum to $0$, then the zeros of $f(z-d)$ are
obtained by adding $d$ to each of the four zeros.  Their sum is
$4d$ so $\Psi_d(V)\subseteq V_{4d}$.  Translation is
invertible, and both spaces have dimension four.
\end{proof}

We note that $\Psi_d$ does not depend on $\eta$.

The only $\Psi_d$ that appears in the rest of the paper is $\Psi_{1/8}$ so we will denote it by $\Psi$.  Thus,
$\Psi:V\to\Theta_4(\tau)$ is the linear isomorphism
$(\Psi f)(z)=f(z-\tfrac18)$.

\subsection{Notation}

Recall that $e_0,e_1,e_2,e_3$ denote the vertices of the singular
quadrics in $\PP V^*$ that contain $E_\tau$.

For all $\eta\in\CC$, we define
\begin{align*}
(t_\eta:E_\tau\to E_\tau)&:=(z\mapsto z+\overline\eta),
\\
\Gamma_\eta & : =\{(z,z+\overline\eta)\mid z\in E_\tau\}
\subset\PP V^*\times\PP V^*, 
\\
K_\eta & :=\{r\in V^{\otimes2}\mid r|_{\Gamma_\eta}=0\}, 
\\
R_\eta & :=\{r\in V^{\otimes2}\mid r|_{\Gamma_\eta}=0\ \text{and}\
r(e_i,e_i)=0\ \text{for }i=0,1,2,3\},
\\
A(\eta)& :=TV/(R_\eta),
\\
B(\eta)& :=TV/(K_\eta),
\\
D&:=\operatorname{span}\{\vartheta_{ab}(2z_1)\vartheta_{ab}(2z_2)
\mid a,b\in\{0,1\}\}\subset V^{\otimes2},
\\
C_\eta&:=D\cap K_\eta.
\end{align*}
Each of $\Gamma_\eta$, $K_\eta$, $R_\eta$, $C_\eta$, $A(\eta)$,
and $B(\eta)$ depends only on the image of $\eta$ in $E_\tau$;
$D$ is independent of $\eta$, so the image of $C_\eta$ in
$A(\eta)_2$ depends only on $\overline\eta$.

Let $B(E,\sigma,\mathcal L)$ denote the twisted homogeneous coordinate
ring as defined in \cite{AVdB}.\footnote{Feigin
and Odesskii also defined these algebras in their Kiev preprint
\cite{FeiginOdesskiiKiev} and in
\cite[\S3.4]{FeiginOdesskiiSklyanin}.}  The next result implies that
\[
B(\eta)\cong B(E_\tau,t_\eta,\mathcal O_{E_\tau}(4[0])).
\]

\begin{proposition}\label{prop.graph-relations-give-THCR}
\cite[Thm. 3.4]{SmithStafford}
Let $E$ be an elliptic curve, let $\sigma$ be a translation automorphism
of $E$, and let $\mathcal L$ be a line bundle on $E$ of degree at least
four.  Put $V=H^0(E,\mathcal L)$, and embed $E$ in
$\PP V^*$ by the complete linear system $|\mathcal L|$.  Define
\[
 \Gamma_\sigma=\{(p,\sigma(p))\mid p\in E\}
 \subset\PP V^*\times\PP V^*,
 \qquad 
 K_\sigma=\{r\in V^{\otimes2}\mid r|_{\Gamma_\sigma}=0\}.
\]
Then
\[
 TV/(K_\sigma)\cong B(E,\sigma,\mathcal L),
\]
where $(K_\sigma)$ denotes the two-sided ideal generated by $K_\sigma$.
\end{proposition}

Clearly, $R_0=\bigwedge^2V$ and $A(0)=SV$.

We will show that $A(\eta)$ is the $4$-dimensional Sklyanin
algebra when $\overline{\eta}\notin E_\tau[2]$
(Theorem~\ref{prop:addition-formulas-relations}).

\subsection{Quartic elliptic normal curves in
\texorpdfstring{$\PP^3$}{P3}}

Fix $c\in\CC$ and let $E_c$ denote the image of $E_\tau$ in
$\PP V_c^*$ under the embedding defined by the complete linear system
associated to $V_c$, i.e., the complete linear system
$|\mathcal L_c|$.  Thus $E_c$ is an elliptic normal quartic; because
every hyperplane in $\PP V_c^*$ is the zero locus of a nonzero function
in $V_c$, every hyperplane in $\PP V_c^*$ meets $E_c$ with multiplicity
four and the sum of those four points equals $c$ in $E_\tau$.
Conversely, it follows from \eqref{eq.Vc.zeros} that if four points on
$E_c$, counted with multiplicity, sum to $c$ in $E_\tau$, then there is
a hyperplane that meets $E_c$ at exactly those four points.

If $L$ is a secant line to $E_c$, then the scheme-theoretic intersection
$L\cap E_c$ is a divisor of degree two, say $(p)+(q)$, and we write
$L=\overline{pq}$.  This notation includes the case $p=q$, when $L$ is
the tangent line at $p$.

When working in $\PP V_c^*$, we identify $E_\tau$ with its image $E_c$.

\begin{proposition}
\label{prop.geometric.vertices}
The pencil of quadrics containing $E_c$ can be labelled
$Q_c(w)$, $w\in E_\tau$, in such a way that
\[
 Q_c(w)=\bigcup_{z\in E_\tau}\overline{z,w-z}.
\]
Thus
\[
 Q_c(w)=Q_c(c-w),
\]
and there are no other equalities among these quadrics.  The singular
members of the pencil are precisely the four quadrics $Q_c(w)$ for
which $2w=c$ in $E_\tau$.

The lines that lie on these quadrics are precisely the secant lines to
$E_c$.

If $p\in\PP V_c^*-E_c$, then $p$ lies on a unique quadric in
the pencil.  If that quadric is smooth, exactly two lines on
the quadric pass through $p$.  If the quadric is not smooth, then it has
rank three (because $E_c$ does not lie on a plane), so every line on it
passes through its vertex.  Thus the four vertices are precisely
the points in $\PP V_c^*-E_c$ that lie on infinitely many secant lines.
\end{proposition}

\begin{proof}
Four points of $E_c$, counted with multiplicity, lie on a common
hyperplane in $\PP V_c^*$ if and only if their sum in $E_\tau$ is $c$.
Let $Q$ be a smooth quadric containing $E_c$, and choose a
hyperplane $H$ such that $Q\cap H=L\cup L'$ is a union of two distinct
lines, one from each ruling.  Since an
elliptic normal quartic on a smooth quadric has bidegree $(2,2)$, the
scheme-theoretic intersections $L\cap E_c$ and $L'\cap E_c$ both
have length two.  Write $L\cap E_c=(p)+(q)$ and
$L'\cap E_c=(p')+(q')$.  Because $E_\tau\cap H$ is a hyperplane
section, $p+q+p'+q'=c$.
Thus, if the secants in one ruling have ``endpoint sum'' $w$, those in the
other ruling have endpoint sum $c-w$.  Hence
$Q_c(w)=Q_c(c-w)$.  Conversely, the two rulings determine these two
sums, so there are no further equalities.  The two rulings coincide
exactly when the quadric is singular; this occurs when $w=c-w$, or
equivalently $2w=c$.  The remaining assertions about secants and
vertices follow from the two rulings on a smooth quadric and the unique
vertex of a singular quadric.
\end{proof}

\subsection{The action of the Heisenberg group on
\texorpdfstring{$V_c$}{V-c}}

Let
\[
 H_4=\langle S,T,\epsilon\mid
 S^4=T^4=\epsilon^4=1,\ \epsilon=[S,T],\
 [S,\epsilon]=[T,\epsilon]=1\rangle .
\]
This is the Heisenberg group of order $4^3$.  For every $c\in\CC$,
define its action on $V_c$ by
\begin{equation}
\label{eq.H4.action.on.Vc}
 \begin{aligned}
  (S\cdot f)(z)&=f\!\left(z+\tfrac14\right),\\
  (T\cdot f)(z)&=
  e\!\left(z-\tfrac14c+\tfrac12-\tfrac38\tau\right)
  f\!\left(z+\tfrac14\tau\right).
 \end{aligned}
\end{equation}
and $\epsilon\cdot f=\ii f$.  With these actions $V_c$ becomes an
irreducible $H_4$-module.

\begin{lemma}
\label{lem.H4.action.on.V}
Let $\lambda=-e(-\tfrac12\tau)$. The action of $H_4$ on $V=V_0$ is such that 
\[
 S^2\cdot\vartheta_{ab}(2z)=(-1)^a\vartheta_{ab}(2z),
 \qquad
 T^2\cdot\vartheta_{ab}(2z)=\lambda^2(-1)^b\vartheta_{ab}(2z).
\]
For the contragredient action on $V^*$,    each $\epsilon_{ab}$ is a simultaneous eigenvector for $S^2$ and $T^2$.
\end{lemma}
\begin{proof}
Since the action of $S$ and $T$ on $V=V_0$ is
\begin{align*}
 (S\cdot f)(z)&=f\!\left(z+\tfrac14\right),\\
 (T\cdot f)(z)&=
 e\!\left(z+\tfrac12-\tfrac38\tau\right)
 f\!\left(z+\tfrac14\tau\right), 
\end{align*}
it follows from the translation formulas in \cite[Table~0, p.~19]{MumfordTataI} that $S$ and $T$ act on the
basis $\{\vartheta_{00}(2z),\vartheta_{01}(2z),
\vartheta_{10}(2z),\vartheta_{11}(2z)\}$ for $V$ as
\begin{align*}
 S\cdot\vartheta_{00}(2z)&=\vartheta_{01}(2z),&
 S\cdot\vartheta_{01}(2z)&=\vartheta_{00}(2z),\\
 S\cdot\vartheta_{10}(2z)&=\vartheta_{11}(2z),&
 S\cdot\vartheta_{11}(2z)&=-\vartheta_{10}(2z),\\
 T\cdot\vartheta_{00}(2z)&=\lambda\vartheta_{10}(2z),&
 T\cdot\vartheta_{01}(2z)&=-\ii\lambda\vartheta_{11}(2z),\\
 T\cdot\vartheta_{10}(2z)&=\lambda\vartheta_{00}(2z),&
 T\cdot\vartheta_{11}(2z)&=-\ii\lambda\vartheta_{01}(2z).
\end{align*}
The result follows.
\end{proof}

It follows from the description of the actions of $S$ and $T$ on the $\vartheta_{ab}(2z)$ that
\begin{alignat*}{4}
 S\cdot Y_{00}&=-\ii Y_{01},&\qquad
 S\cdot Y_{01}&=\ii Y_{00},&\qquad
 S\cdot Y_{10}&=\ii Y_{11},&\qquad
 S\cdot Y_{11}&=\ii Y_{10},\\
 T\cdot Y_{00}&=-\ii\lambda Y_{10},&
 T\cdot Y_{01}&=\lambda Y_{11},&
 T\cdot Y_{10}&=\ii\lambda Y_{00},&
 T\cdot Y_{11}&=-\lambda Y_{01}.
\end{alignat*}

As noted in  \cite[Lem.~2.8]{CKS1}, and elsewhere, the formula in
\eqref{eq.H4.action.on.Vc}, with $c=3/2$, gives the action on
$\Theta_4(\tau)$ for which
\[
 S\cdot\theta_\alpha=\ii^\alpha\theta_\alpha,
 \qquad T\cdot\theta_\alpha=\theta_{\alpha+1}.
\]

As noted in \eqref{eq.Vc.zeros},
$V=H^0(E_\tau,\mathcal L_0)$, and
$\Theta_4(\tau)=H^0(E_\tau,\mathcal L_{1/2})$.  Translation by $1/8$
carries $\mathcal L_0$ to $\mathcal L_{1/2}$ and induces the
isomorphism $\Psi_{1/8}$ defined in Lemma \ref{prop.translation.V.Vprime}.
Transporting the action of $H_4$ on $V$ through $\Psi_{1/8}$ gives
an action of $H_4$ on $\Theta_4(\tau)$ that is {\emph not} the same as that obtained by putting $c=3/2$ in \eqref{eq.H4.action.on.Vc}: 
for $h\in H_4$ and
$F\in\Theta_4(\tau)$, define
\[
 h\star F:=\Psi_{1/8}\bigl(h\cdot\Psi_{1/8}^{-1}(F)\bigr).
\]
In particular,
\begin{align*}
 (S\star F)(z)&=F\!\left(z+\tfrac14\right),\\
 (T\star F)(z)&=
 e\!\left(z+\tfrac38-\tfrac38\tau\right)
 F\!\left(z+\tfrac14\tau\right).
\end{align*}
Thus $S\star=S\cdot$ and $T\star=\ii T\cdot$; the two $S$-actions agree,
but the two $T$-actions differ by the scalar $\ii$.  Although the two
actions of $H_4$ on $\Theta_4(\tau)$ differ, they induce the same projective action on
$\PP\Theta_4(\tau)^*$.  By construction, $\Psi_{1/8}$ intertwines the
action on $V$ with the transported action $\star$ on
$\Theta_4(\tau)$.

The functions
\[
u_0=Y_{00}-\ii Y_{01},\quad u_1=Y_{10}+Y_{11},\quad
u_2=Y_{00}+\ii Y_{01},\quad u_3=Y_{10}-Y_{11}
\]
provide a basis for $V$ consisting of eigenvectors for $S\cdot$:
\[
 S\cdot u_\alpha
 =u_\alpha\!\left(z+\tfrac14\right)
   =\ii^\alpha u_\alpha(z).
\]
By \cite[Lem.~2.8]{CKS1}, the $\theta_\alpha$ are also eigenvectors for
$S\cdot$, with these same eigenvalues.
Since the four eigenspaces are one-dimensional, there are unique
nonzero constants $\kappa_\alpha$ such that
\begin{equation}
\label{eq.basis.comparison}
 \begin{aligned}
 \theta_0(z)&=\kappa_0
   \bigl(Y_{00}-\ii Y_{01}\bigr)\!\left(z-\tfrac18\right),\\
 \theta_1(z)&=\kappa_1
   \bigl(Y_{10}+Y_{11}\bigr)\!\left(z-\tfrac18\right),\\
 \theta_2(z)&=\kappa_2
   \bigl(Y_{00}+\ii Y_{01}\bigr)\!\left(z-\tfrac18\right),\\
 \theta_3(z)&=\kappa_3
   \bigl(Y_{10}-Y_{11}\bigr)\!\left(z-\tfrac18\right).
 \end{aligned}
\end{equation}
The constants $\kappa_\alpha$ are given by the formula
\[
 \kappa_\alpha=
 \frac{\theta_\alpha(z_\alpha)}
      {u_\alpha(z_\alpha-\tfrac18)},
\]
where $z_\alpha$ is any point for which the denominator is nonzero.
Consequently,
\begin{equation}
\label{eq.Psi.inverse.on.CKS.basis}
 \Psi_{1/8}^{-1}(\theta_\alpha)=\kappa_\alpha u_\alpha,
 \qquad \alpha\in\ZZ_4.
\end{equation}
Equation \eqref{eq.Psi.inverse.on.CKS.basis} is the vector-space
isomorphism that will be extended to the tensor algebras.

The dual projective isomorphism
\[
 \Psi_{1/8}^*:\PP\Theta_4(\tau)^*\longrightarrow\PP V^*
\]
sends the evaluation point represented by $z\in E_\tau$ to the
evaluation point represented by $z-\tfrac18$.  Thus translation changes
the chosen degree-four line bundle but does not change a displacement
between two points of $E_\tau$.

There is a more intrinsic way to recognize the four vertices from the
Heisenberg action.  We now use the basis in
\cite[Prop.~2.7]{CKS1}.  For $\alpha\in\ZZ_4$, set
\[
 \theta_{\alpha,c}(z)
   =\theta_\alpha\!\left(z-\tfrac14c+\tfrac38\right).
\]
Then $\{\theta_{0,c},\theta_{1,c},\theta_{2,c},\theta_{3,c}\}$ is a
basis for $V_c$.  Proposition~2.7 of CKS1 gives
\begin{align*}
 \theta_{\alpha,c}\!\left(z+\tfrac14\right)
   &=\ii^\alpha\theta_{\alpha,c}(z),\\
 \theta_{\alpha,c}\!\left(z+\tfrac14\tau\right)
   &=-e\!\left(-z+\tfrac14c+\tfrac38\tau\right)
       \theta_{\alpha+1,c}(z).
\end{align*}
With the $H_4$-action on $V_c$ defined in
\eqref{eq.H4.action.on.Vc}, these identities
give
\begin{equation}
\label{eq.H4.action.on.theta.alpha.c}
 S\cdot\theta_{\alpha,c}=\ii^\alpha\theta_{\alpha,c},
 \qquad
 T\cdot\theta_{\alpha,c}=\theta_{\alpha+1,c}.
\end{equation}

\begin{proposition}
\label{prop.vertices.Heisenberg}
The vertices of the four singular quadrics containing $E_c$ are the four common fixed points of the
projective transformations induced by $S^2$ and $T^2$.  Equivalently,
they are the four simultaneous eigenlines for $S^2$ and $T^2$ on
$V_c^*$.

They are not, in general, the fixed points of $S$.  The element $S$
permutes the singular quadrics, and hence permutes their vertices.
\end{proposition}

\begin{proof}
The projective action of $S$ on the embedded elliptic curve is
translation by the point $a=-1/4$.  Indeed, the action on $V_c^*$ is
contragredient to the action on $V_c$.  Translation by $a$ sends the secant
line $\overline{x,w-x}$, which lies on $Q_c(w)$, to
\[
 \overline{x+a,w-x+a},
\]
which lies on $Q_c(w+2a)$.  Thus $S$ sends $Q_c(w)$ to
$Q_c(w+2a)$.
Since $2a$ is a nonzero point of order two, $S$ permutes the four
singular quadrics.  A similar
argument applies to $T$.

On the other hand, $S^2$ acts by translation by $2a$.  It therefore
sends $Q_c(w)$ to $Q_c(w+4a)=Q_c(w)$.  It therefore preserves every
member of the pencil.  In particular it fixes the unique vertex of each singular
quadric.  Similarly, $T^2$ fixes all four vertices.

The operators $S^2$ and $T^2$ commute and have four one-dimensional
simultaneous eigenspaces in the irreducible four-dimensional
Heisenberg representation.  Their common projective fixed locus
therefore consists of exactly four points.  Since it contains the four
vertices, it is precisely the set of vertices.
\end{proof}

\begin{corollary}\label{cor.eij.are.vertices}
The points $e_{ij}$, $i,j\in\{0,1\}$, are the vertices of the four
singular quadrics containing $E_\tau\subset\PP V^*$.
\end{corollary}

\begin{proof}
The formulas above show that the four lines $\CC\epsilon_{ij}$ are the
simultaneous eigenlines for $S^2$ and $T^2$ on $V^*$.  The result now
follows from Proposition~\ref{prop.vertices.Heisenberg} with $c=0$.
\end{proof}

\begin{lemma}\label{lem.H4.stability}
For every $\eta\in\CC$, the subspaces
$D$, $K_\eta$, $R_\eta$, and $C_\eta$ of $V^{\otimes2}$ are
$H_4$-submodules.
\end{lemma}

\begin{proof}
The formulas in Lemma~\ref{lem.H4.action.on.V} show that $S$ and $T$
permute, up to nonzero scalar multiples, the four tensors
$\vartheta_{ab}(2z_1)\vartheta_{ab}(2z_2)$ that span $D$.  Hence $D$
is $H_4$-stable.  The projective transformations induced by $S$ and
$T$ on $E_\tau\subset\PP V^*$ are translations by points of
$E_\tau[4]$.  Their diagonal actions on
$E_\tau\times E_\tau$ therefore preserve $\Gamma_\eta$, so $K_\eta$
is $H_4$-stable.  Proposition~\ref{prop.vertices.Heisenberg} shows
that these transformations also permute the four vertices.  It follows
from the definition of $R_\eta$ that $R_\eta$ is $H_4$-stable.  Finally,
$C_\eta=D\cap K_\eta$, so $C_\eta$ is $H_4$-stable as well.
\end{proof}

Let $\xi_0,\xi_1,\xi_2,\xi_3$ be the basis for $V_c^*$ dual to
$\theta_{0,c},\theta_{1,c},\theta_{2,c},\theta_{3,c}$, and write
$[x_0:x_1:x_2:x_3]$ for the corresponding homogeneous coordinates on
$\PP V_c^*$.  Thus the coordinate functions $x_\alpha$ are identified
with the basis elements $\theta_{\alpha,c}$ of $V_c$.  The preceding
proposition now gives the vertices
without writing equations for the quadrics:
\begin{equation}
\label{eq.vertices.CKS.Prop.2.7.basis}
 \begin{aligned}
 e_{0,c}&=[1:0:1:0],
 &e_{1,c}&=[1:0:-1:0],\\
 e_{2,c}&=[0:1:0:1],
 &e_{3,c}&=[0:1:0:-1].
 \end{aligned}
\end{equation}
Indeed, by \eqref{eq.H4.action.on.theta.alpha.c}, $S^2$ acts on
$\xi_\alpha$ with eigenvalue $(-1)^\alpha$, so its even and odd
eigenspaces on $V_c^*$ are $\langle\xi_0,\xi_2\rangle$ and
$\langle\xi_1,\xi_3\rangle$.  The same formula shows, after passing
to the contragredient action, that $T^2$ interchanges $\xi_0$ with
$\xi_2$ and $\xi_1$ with $\xi_3$.
The four points
in \eqref{eq.vertices.CKS.Prop.2.7.basis} are therefore exactly their
four simultaneous eigenlines.  Notice that the formula is independent
of $c$; the dependence on $c$ is carried by the basis functions
$\theta_{\alpha,c}$ themselves.

Specialize now to $c=3/2$, so $V_c=\Theta_4(\tau)$ and
$\theta_{\alpha,3/2}=\theta_\alpha$.  Thus the four vertices in CKS
coordinates are obtained from
\eqref{eq.vertices.CKS.Prop.2.7.basis} by setting $c=3/2$.
Translation of the line bundle carries the pencil of quadrics containing
this copy of $E_\tau$ to the pencil of quadrics containing the copy of
$E_\tau$ in $\PP V^*$ and,
after a suitable ordering, gives
\begin{equation}
\label{eq.vertices.transport}
 \Psi^*(e_{\alpha,3/2})=e_\alpha,
 \qquad \alpha \in \ZZ_4.
\end{equation}
Here the $e_\alpha$ are the vertices used in the definition of $R_\eta$;
Proposition~\ref{prop.geometric.vertices} characterizes them by secant
lines, and Proposition~\ref{prop.vertices.Heisenberg} characterizes them
as the simultaneous eigenlines for $S^2$ and $T^2$.

\subsection{Jacobi's addition theorems}

The theta identities used below are specializations of Jacobi's
fundamental addition theorem for products of four theta functions.
Jacobi developed elementary addition formulas in the
\emph{Fundamenta Nova}~\cite[\S\S54, 61]{JacobiFundamenta} and treated
the general quartic identity in his lectures on elliptic functions,
published posthumously in \cite[pp.~497--538, especially
pp.~505--507]{JacobiTheta}.  For a modern derivation of the addition
formulas from Jacobi's five-term quartic identity, see
Mumford~\cite[Chapter~I, \S5]{MumfordTataI}.  Koornwinder
\cite{Koornwinder} proves that Weierstrass's three-term identity
\eqref{eq.Weierstrass.three-term} and Jacobi's five-term identity
\eqref{eq.Jacobi.five-term.four-theta} are equivalent, and discusses
their history and use.

The addition identities express products in the sum and
difference variables as linear combinations of diagonal products.  If
$a,b\in W$, there are constants $c_{rs}(a,b)$ such that
\[
 a(z_1-z_2)b(z_1+z_2)
 =\sum_{r,s\in\{0,1\}}c_{rs}(a,b)
   \vartheta_{rs}(2z_1)\vartheta_{rs}(2z_2).
\]
In particular, $a(z_1-z_2)b(z_1+z_2)\in D$.
Conversely, the same identities show that every element of $D$ is a
linear combination of products of this form.  Hence they induce an
isomorphism
\begin{equation}\label{eq.addition.isomorphism}
 W(u)\otimes W(v)\xrightarrow{\sim}D,
 \qquad
 a(u)\otimes b(v)\longmapsto a(z_1-z_2)b(z_1+z_2).
\end{equation}
This isomorphism is used in the proof of
Proposition~\ref{thm.dim.Ceta} to determine $C_\eta$.
Explicit formulas can be found in
Kharchev--Zabrodin~\cite[Sections~3.2 and~4.2]{KZ}.

The next four identities are
(4.7a), (4.8a), (4.9a), and (4.10a) in \cite{KZ}:
\begin{align}
\theta_2(0)^2\theta_1(u+v)\theta_1(u-v)
 &=\phantom{-}\theta_2(v)^2\theta_1(u)^2
   -\theta_1(v)^2\theta_2(u)^2,\label{eq:add1}\\
\theta_2(0)^2\theta_2(u+v)\theta_2(u-v)
 &=-\theta_1(v)^2\theta_1(u)^2
   +\theta_2(v)^2\theta_2(u)^2,\label{eq:add2}\\
\theta_2(0)^2\theta_3(u+v)\theta_3(u-v)
 &=\phantom{-}\theta_4(v)^2\theta_1(u)^2
   +\theta_3(v)^2\theta_2(u)^2,\label{eq:add3}\\
\theta_2(0)^2\theta_4(u+v)\theta_4(u-v)
 &=\phantom{-}\theta_3(v)^2\theta_1(u)^2
   +\theta_4(v)^2\theta_2(u)^2.\label{eq:add4}
\end{align}
These identities are used in the calculation immediately
preceding Theorem~\ref{thm:explicit-C}.

Assume $\overline\eta\notin E_\tau[2]$.  In this case the elements
$X_{ab}^{\otimes2}$ form a basis for $D$.  We rewrite the identities
above in terms of this basis.  Put
\[
 u=z_1+z_2,\qquad v=z_1-z_2.
\]
Since $u+v=2z_1$ and $u-v=2z_2$, equations
\eqref{eq:add1}--\eqref{eq:add4} give
\begin{align*}
X_{11}(z_1)X_{11}(z_2)
 &=\frac{\theta_1(\eta)^2}{\theta_2(0)^2}
   \bigl(\theta_2(v)^2\theta_1(u)^2
         -\theta_1(v)^2\theta_2(u)^2\bigr),\\
X_{10}(z_1)X_{10}(z_2)
 &=\frac{\theta_2(\eta)^2}{\theta_2(0)^2}
   \bigl(\theta_1(v)^2\theta_1(u)^2
         -\theta_2(v)^2\theta_2(u)^2\bigr),\\
X_{00}(z_1)X_{00}(z_2)
 &=\frac{\theta_3(\eta)^2}{\theta_2(0)^2}
   \bigl(\theta_4(v)^2\theta_1(u)^2
         +\theta_3(v)^2\theta_2(u)^2\bigr),\\
X_{01}(z_1)X_{01}(z_2)
 &=-\frac{\theta_4(\eta)^2}{\theta_2(0)^2}
   \bigl(\theta_3(v)^2\theta_1(u)^2
         +\theta_4(v)^2\theta_2(u)^2\bigr).
\end{align*}
These four identities give explicitly the change of basis between the
bases
\[
 \{\vartheta_{00}(2z_1)\vartheta_{00}(2z_2),
   \vartheta_{01}(2z_1)\vartheta_{01}(2z_2),
   \vartheta_{10}(2z_1)\vartheta_{10}(2z_2),
   \vartheta_{11}(2z_1)\vartheta_{11}(2z_2)\}
\]
and
\[
 \{\theta_1(v)^2\theta_1(u)^2,
   \theta_1(v)^2\theta_2(u)^2,
   \theta_2(v)^2\theta_1(u)^2,
   \theta_2(v)^2\theta_2(u)^2\}
\]
for $D$ under the isomorphism in \eqref{eq.addition.isomorphism}.

The proof of Theorem~\ref{thm:R-eta-basis} uses the following two
families of mixed addition identities:
\begin{subequations}\label{eq.mixed.addition.identities}
\begin{align}
\vartheta_{ab}(0)\vartheta_{ij}(0)
 \vartheta_{kl}(u+v)\vartheta_{11}(u-v)
&=\vartheta_{11}(u)\vartheta_{kl}(u)
  \vartheta_{ab}(v)\vartheta_{ij}(v)\notag\\
&\quad-\vartheta_{ab}(u)\vartheta_{ij}(u)
  \vartheta_{kl}(v)\vartheta_{11}(v),
\label{eq.mixed.addition.first}\\
\vartheta_{ab}(0)\vartheta_{ij}(0)
 \vartheta_{ab}(u+v)\vartheta_{ij}(u-v)
&=\vartheta_{ab}(u)\vartheta_{ij}(u)
  \vartheta_{ab}(v)\vartheta_{ij}(v)\notag\\
&\quad+(-1)^{k+l}\vartheta_{kl}(u)\vartheta_{11}(u)
  \vartheta_{kl}(v)\vartheta_{11}(v).
\label{eq.mixed.addition.second}
\end{align}
\end{subequations}
In \eqref{eq.mixed.addition.first} and
\eqref{eq.mixed.addition.second}, take $(ab,ij,kl)$ successively to be
\[
 (00,01,10),\qquad (01,10,00),\qquad (10,00,01).
\]
The two displayed families thereby give the six identities used in the
proof of Theorem~\ref{thm:R-eta-basis}.  In numbered theta
notation these are equations (4.11a)--(4.11f) in
Kharchev--Zabrodin~\cite{KZ}, translated from numbered theta notation
to characteristic notation.

We will also use Jacobi's five-term four-theta identity.  Given
$x_1,x_2,x_3,x_4\in\CC$, define
\begin{align*}
 y_1&=\tfrac12(x_1+x_2+x_3+x_4),&
 y_2&=\tfrac12(x_1+x_2-x_3-x_4),\\
 y_3&=\tfrac12(x_1-x_2+x_3-x_4),&
 y_4&=\tfrac12(x_1-x_2-x_3+x_4).
\end{align*}
Then
\begin{equation}\label{eq.Jacobi.five-term.four-theta}
 2\prod_{r=1}^4\vartheta_{11}(x_r)
 =\prod_{r=1}^4\vartheta_{11}(y_r)
     -\prod_{r=1}^4\vartheta_{10}(y_r)
     +\prod_{r=1}^4\vartheta_{00}(y_r)
     -\prod_{r=1}^4\vartheta_{01}(y_r).
\end{equation}

We will also use Weierstrass's three-term identity
\begin{equation}\label{eq.Weierstrass.three-term}
 \theta(z\mathbin\pm\alpha)\theta(\beta\mathbin\pm\gamma)
 +\theta(z\mathbin\pm\beta)\theta(\gamma\mathbin\pm\alpha)
 +\theta(z\mathbin\pm\gamma)\theta(\alpha\mathbin\pm\beta)=0.
\end{equation}

\subsection{Isomorphisms and anti-isomorphisms}

\begin{proposition}\label{prop:A-under-elliptic-automorphism}
Let $\mu:E_\tau\to E_\tau$ be an isomorphism of elliptic curves, so
$\mu(0)=0$, and choose a linear lift of the projective automorphism of
$\PP V^*$ induced by $\mu$.  Let $\eta,\eta_1,\eta_2\in\CC$.
\begin{enumerate}
\item If $\overline{\eta_2}=\mu(\overline{\eta_1})$, then there are
graded algebra isomorphisms
\[
 A(\eta_1)\cong A(\eta_2)
 \qquad\text{and}\qquad
 A(-\eta_2)\cong A(\eta_1)^{\mathrm{op}}.
\]
\item There is a canonical graded algebra isomorphism
\[
 A(-\eta)\cong A(\eta)^{\mathrm{op}}.
\]
\end{enumerate}
\end{proposition}

\begin{proof}
For the first isomorphism in (1), put
$\mathcal L=\mathcal O_{E_\tau}(4[0])$.  Since $\mu(0)=0$, one has
$\mu^*\mathcal L\cong\mathcal L$.  Hence $\mu$ extends to a projective
automorphism $\widetilde\mu$ of $\PP V^*$; choose
$\Phi_\mu:V\to V$ so that the projective transformation induced by
$(\Phi_\mu^{-1})^*$ is $\widetilde\mu$.  The map
$\widetilde\mu\times\widetilde\mu$ carries
\[
 (z,z+\overline{\eta_1})
 \longmapsto
 \bigl(\mu(z),\mu(z)+\mu(\overline{\eta_1})\bigr),
\]
and therefore carries $\Gamma_{\eta_1}$ to $\Gamma_{\eta_2}$.  It also
permutes the four singular quadrics containing $E_\tau$, and hence
their four vertices.  Consequently,
$(\Phi_\mu\otimes\Phi_\mu)(R_{\eta_1})=R_{\eta_2}$, which gives
$A(\eta_1)\cong A(\eta_2)$.

To prove (2), let $\sigma:V^{\otimes2}\to V^{\otimes2}$ be the flip
$\sigma(u\otimes v)=v\otimes u$.  If $r$ is regarded as a form on
$\PP V^*\times\PP V^*$, then $\sigma(r)(p,q)=r(q,p)$, so the same graph
and vertex calculation gives $\sigma(R_\eta)=R_{-\eta}$.
The identity map from $V$ to the degree-one subspace of $(TV)^{\mathrm{op}}$
extends uniquely to an algebra isomorphism $TV\to(TV)^{\mathrm{op}}$;
on $V^{\otimes2}$ it is $\sigma$.  It therefore induces
\[
 A(\eta)\cong A(-\eta)^{\mathrm{op}}.
\]
Replacing $\eta$ by $-\eta$ gives (2).  Finally, (2), applied to
$\eta_2$, and the first isomorphism in (1) give
$A(-\eta_2)\cong A(\eta_1)^{\mathrm{op}}$.
\end{proof}

\section{Relations between the algebras
\texorpdfstring{$A(\eta)$, $S(\eta)$,
$Q_{4,1}(-\tfrac12\eta)$, and $\mathcal R(\eta)$}
{A(eta), S(eta), Q41(-eta/2), and R(eta)}}

Recall that $S(\eta)$ is defined only when
$\overline\eta\notin E_\tau[2]$.

In this section we prove the results in (1-1).

\subsection{The complementary subspace \texorpdfstring{$C_\eta$}{C-eta}}

The space $C_\eta$ has three important properties:
\begin{enumerate}
\item $K_\eta=C_\eta\oplus R_\eta$, hence the term
``complementary subspace.''  This is proved as part of the proof that
$A(\eta)\cong S(\eta)$ when $\overline\eta\notin E_\tau[2]$
(Theorem~\ref{prop:addition-formulas-relations}); the decomposition for
all $\eta$ is recorded in Corollary~\ref{cor:Keta.decomposition}.
\item $C_\eta$ is canonical in the sense that it is the only complement
to $R_\eta$ in $K_\eta$ that is an $H_4$-submodule of $K_\eta$; see
Proposition~\ref{prop.H4.unique.complement}.
\item The image of $C_\eta$ in $A(\eta)_2$ belongs to the center of
$A(\eta)$; see Theorem~\ref{thm:Ceta-central-all}.
\end{enumerate}
The significance of $C_\eta$ was first recognized by
Rains--Ruijsenaars~\cite[Appendix~B]{RR}.

Restriction to $\Gamma_\eta$ is the linear map
\[
 V^{\otimes2}\longrightarrow
 H^0\bigl(\Gamma_\eta,\mathcal O(1,1)|_{\Gamma_\eta}\bigr)
\]
whose kernel is $K_\eta$.  The geometric dimension calculation in
Proposition~\ref{prop.dim.Keta} below
shows that this map is surjective.  Consequently,
\[
 H^0\bigl(\Gamma_\eta,\mathcal O(1,1)|_{\Gamma_\eta}\bigr)
 \cong V^{\otimes2}/K_\eta.
\]
Hence $C_\eta=D\cap K_\eta={\rm ker}\bigl(D\longrightarrow
H^0(\Gamma_\eta,\mathcal O(1,1)|_{\Gamma_\eta})\bigr)$.
In theta-function terms,
the target is the space of restrictions to the graph: after identifying
$\Gamma_\eta$ with $E_\tau$ by $z\mapsto(z,z+\overline\eta)$, an element
$r\in D$ maps to the order-eight theta function
$z\mapsto r(z,z+\overline\eta)$,
which is a section of $\mathcal L\otimes t_\eta^*\mathcal L$.

\begin{proposition}\label{thm.dim.Ceta}
For all $\eta\in\CC$, the following statements hold.
\begin{enumerate}
\item $\dim C_\eta=2$.
\item $C_\eta=\theta(z_1-z_2\pm\eta)W(z_1+z_2)$.
\item $C_\eta=\{c\,\theta(z_1-z_2\pm\eta)
\theta(z_1+z_2\pm\gamma)\mid c,\gamma\in\CC\}$.
\item Our $C_\eta$ is the subspace denoted by $C_\eta$ in
Rains--Ruijsenaars~\cite[Appendix~B]{RR}.
\end{enumerate}
\end{proposition}

\begin{proof}
The restriction of a function $f(z_1,z_2)\in V^{\otimes2}$ to
$\Gamma_\eta$ is the function $z\mapsto f(z,z+\eta)$.  Under the
addition-theorem isomorphism \eqref{eq.addition.isomorphism}, where
$u=z_1-z_2$ and $v=z_1+z_2$, the restriction of $a(u)b(v)\in D$ is
$a(-\eta)b(2z+\eta)$.  Hence
\[
C_\eta=\ker\bigl(W(u)\xrightarrow{\operatorname{ev}_{-\eta}}\CC\bigr)
\otimes W(v).
\]
The kernel of evaluation at $-\eta$ is one-dimensional and is spanned
by $\theta(u\pm\eta)$.  Since $\dim W(v)=2$, it follows that
$\dim C_\eta=2$ and
$C_\eta=\theta(u\pm\eta)W(v)
=\theta(z_1-z_2\pm\eta)W(z_1+z_2)$.
This proves (1) and (2).  The notation introduced after
\eqref{eq.definition.Xab} gives
\[
 W(z_1+z_2)
 =\{c\,\theta(z_1+z_2\pm\gamma)\mid c,\gamma\in\CC\}.
\]
This proves (3).  The description in (3) is exactly the subspace
that Rains--Ruijsenaars denote by $C_\eta$ in
\cite[Appendix~B]{RR}, proving (4).
\end{proof}

\begin{proposition}\label{prop:Cnu-Ceta}
If $\nu\in\CC$ and $\eta\not\equiv\pm\nu\pmod{\Lambda_\tau}$, then
$C_\nu\cap C_\eta=\{0\}$ and $D=C_\nu\oplus C_\eta$.
\end{proposition}

\begin{proof}
By Proposition~\ref{thm.dim.Ceta},
$C_\nu=\theta(u\pm\nu)W(v)$ and
$C_\eta=\theta(u\pm\eta)W(v)$.
The zero divisors of the two functions in the first factors are
$(\nu)+(-\nu)$ and $(\eta)+(-\eta)$.  If these divisors were equal,
then $\eta=\nu$ or $\eta=-\nu$ in $E_\tau$.  Under the stated
hypothesis the one-dimensional subspaces
$\CC\theta(u\pm\nu)$ and $\CC\theta(u\pm\eta)$ of $W(u)$ are distinct,
proving both assertions.
\end{proof}

\begin{proposition}\label{prop:intersection-R-C}
For every $\eta\in\CC$, $C_\eta\cap R_\eta=\{0\}$.
\end{proposition}

\begin{proof}
Write an element $r\in D$ in the form
\[
r=\sum_{a,b\in\{0,1\}}c_{ab}
 \vartheta_{ab}(2z_1)\vartheta_{ab}(2z_2).
\]
Recall from the paragraph preceding the definition of $A(\eta)$
that $e_i=[\epsilon_i]$ and that the $\epsilon_i$ are dual to the ordered
basis $(\vartheta_{00}(2z),\vartheta_{01}(2z),
\vartheta_{10}(2z),\vartheta_{11}(2z))$ for $V$. Thus evaluation at the
four pairs $(e_i,e_i)$ gives, in this order, $(c_{00},c_{01},c_{10},c_{11})$. 
If $r\in C_\eta\cap R_\eta$, the defining conditions for $R_\eta$ give
$r(e_i,e_i)=0$ for every $i$. Hence every $c_{ab}$ is zero, so $r=0$.
\end{proof}

\begin{remark}
In \cite[Appendix~B]{RR}, Rains--Ruijsenaars prove that 
$R^{\rm RR}_\eta\cap C_\eta=\{0\}$ where $R^{\rm RR}_\eta$ is the space of relations for their algebra ${\mathcal R}(\eta)$ \cite[(B.16), as a consequence of Lemma B.2]{RR}. 
In \S\ref{sect.RR-algebra}, we will show that $R^{\rm RR}_\eta=R_\eta$.
Thus Proposition~\ref{prop:intersection-R-C} agrees with their result.
Nevertheless it is noteworthy
that the paths to that conclusion are very different: the argument in \cite{RR} is analytic whereas the geometric definition of $A(\eta)$ reduces the problem to four elementary evaluations at the points $(e_i,e_i)$, and since $r \in D$, this 
reduces to the evaluations $\pi(r)(e_i)$.\footnote{See \S\ref{sect.RR-algebra} for the definition of  $R^{\rm RR}_\eta$: it is the span of an explicitly parametrized family of theta functions
$F_\eta(a,b,c_1,c_2)(z_1,z_2)\in V^{\otimes2}$. The proof of \cite[Lemma B.2]{RR} uses Sklyanin's homomorphism
$\mathcal A_\nu$ from $TV$ to an algebra of analytic difference operators which is implicit in \cite[(26)]{Sklyanin1983}; they show that $R^{\rm RR}_\eta \subset {\rm ker}(\mathcal A_\nu)$. 
A key step in their argument is analytic (their Lemma B.2, whose proof they say is {\it not obvious}); on the other hand, $\mathcal A_\nu$ sends the element
$
 \theta(z_1-z_2\pm\eta)\theta(z_1+z_2\pm\gamma)
$
in $C_\eta$ to the scalar $2\theta(2\nu+\eta\pm\gamma)$, which is
nonzero for generic $\nu$; thus $C_\eta \cap {\rm ker}(\mathcal A_\nu)=0$. Thus Rains--Ruijsenaars use
Sklyanin's analytic-difference-operator representation to prove the intersection result as part of a broader analysis of that
representation.}
\end{remark}

\begin{proposition}
\label{prop.dim.Keta}
For every $\eta\in\CC$, $\dim K_\eta=8$.
\end{proposition}

\begin{proof}
Put $\mathcal L=\mathcal O_{E_\tau}(4[0])$.  Under the identification
$\Gamma_\eta\cong E_\tau$ given by $(z,z+\eta)\leftrightarrow z$,
restriction to $\Gamma_\eta$ of the bidegree-$(1,1)$ forms in
$H^0(\PP V^*\times\PP V^*,\mathcal O(1,1))=V^{\otimes2}$ is the
multiplication map
\[
H^0(E_\tau,\mathcal L)\otimes H^0(E_\tau,t_\eta^*\mathcal L)
\longrightarrow H^0(E_\tau,\mathcal L\otimes t_\eta^*\mathcal L).
\]
This map is surjective. Here is a geometric verification. If
$t_\eta^*\mathcal L\not\cong\mathcal L$, choose a base-point-free pencil
$U\subset H^0(E_\tau,\mathcal L)$.  Thus $U$ is a two-dimensional space
of sections of $\mathcal L$ having no common zero; equivalently, the
evaluation map $U\otimes\mathcal O_{E_\tau}\to\mathcal L$ is surjective.
Its kernel is $\mathcal L^{-1}$.  Tensoring this evaluation sequence by
$t_\eta^*\mathcal L$ gives the ``base-point-free pencil trick''
\[
0\longrightarrow t_\eta^*\mathcal L\otimes\mathcal L^{-1}
\longrightarrow U\otimes t_\eta^*\mathcal L
\longrightarrow\mathcal L\otimes t_\eta^*\mathcal L\longrightarrow0.
\]
The first line bundle is nontrivial of degree $0$, so both its zeroth and
first cohomology vanish. Taking global sections therefore proves
surjectivity. If $t_\eta^*\mathcal L\cong\mathcal L$, the map is the quadratic
multiplication map for the elliptic quartic defined by $|\mathcal L|$. An elliptic
quartic is the complete intersection of two quadrics in $\PP^3$; hence
the image has dimension
\[
\dim S^2H^0(E_\tau,\mathcal L)-2=10-2=8,
\]
and is again all of $H^0(E_\tau,\mathcal L^2)$.

Finally, $\mathcal L\otimes t_\eta^*\mathcal L$ has degree $8$, so
$h^0(E_\tau,\mathcal L\otimes t_\eta^*\mathcal L)=8$.
Since $\dim V^{\otimes2}=16$, the kernel of the restriction map has
dimension $16-8=8$.
\end{proof}

\subsection{Proof that \texorpdfstring{$A(\eta)\cong S(\eta)$}
{A(eta) is isomorphic to S(eta)}}

When $\overline{\eta}\notin E_\tau[2]$ and  $(ab,ij,kl)$ is a cyclic permutation of $(00,01,10)$, we define
\begin{equation}\label{eq.alpha.ab}
\alpha_{ab}=(-1)^{a+b}
\left(
\frac{\vartheta_{ab}(\eta)\vartheta_{11}(\eta)}
     {\vartheta_{ij}(\eta)\vartheta_{kl}(\eta)}
\right)^{\!2}.
\end{equation}

\begin{theorem}\label{thm:R-eta-basis}
\label{prop:addition-formulas-relations}
The following statements hold.
\begin{enumerate}
\item If $\overline\eta\notin E_\tau[2]$, then $R_\eta$ has a basis
consisting of the six elements
\[
 \begin{cases}
 [X_{11},X_{ab}]-\alpha_{ab}\{X_{ij},X_{kl}\},\\
 \{X_{11},X_{ab}\}-[X_{ij},X_{kl}],
 \end{cases}
\]
as $(ab,ij,kl)$ ranges over the three cyclic permutations of
$(00,01,10)$.  Hence
\[
 A(\eta)\cong S(\eta)
\]
via the isomorphism that sends the class $x_{ab}$ of $X_{ab}$ to the
standard generator $x_{ab}$ of $S(\eta)$.

\item Suppose $\overline\eta\in E_\tau[2]$.  If
$\overline\eta=0$, then a basis for $R_\eta$ consists of the six
commutators of distinct elements in the basis
$\{Y_{00},Y_{01},Y_{10},
Y_{11}\}$.

If $\overline\eta\ne0$, let $V_+$ and $V_-$ be the two
two-dimensional eigenspaces for a lift to $V$ of translation by
$-\overline\eta$.  Up to interchanging the signs, they are
\[
\begin{array}{c|c|c}
\overline\eta&V_+&V_-\\ \hline
\frac12&
\Span\{Y_{00},Y_{01}\}&
\Span\{Y_{10},Y_{11}\}\\
\frac12\tau&
\Span\{Y_{00},Y_{10}\}&
\Span\{Y_{01},Y_{11}\}\\
\frac12(1+\tau)&
\Span\{Y_{00},Y_{11}\}&
\Span\{Y_{01},Y_{10}\}.
\end{array}
\]
If $(u_1,u_2)$ and $(v_1,v_2)$ are the displayed bases for $V_+$
and $V_-$, respectively, then
\[
 [u_1,u_2],\quad [v_1,v_2],\quad
 \{u_i,v_j\}\quad(1\leq i,j\leq2)
\]
is a basis for $R_\eta$.

\item For every $\eta\in\CC$ and every cyclic permutation
$(ab,ij,kl)$ of $(00,01,10)$, write
$\vartheta_{pq}=\vartheta_{pq}(\eta)$, put
$\bar a=1-a$ and $\bar b=1-b$, and define
\begin{align*}
\rho_{\bar a\bar b}(\eta)&:=\vartheta_{ij}\vartheta_{kl}[Y_{11},Y_{ab}]
-(-1)^{a+b}\vartheta_{ab}\vartheta_{11}\{Y_{ij},Y_{kl}\},\\
\rho'_{\bar a\bar b}(\eta)&:=\vartheta_{11}\vartheta_{ab}\{Y_{11},Y_{ab}\}
-\vartheta_{ij}\vartheta_{kl}[Y_{ij},Y_{kl}].
\end{align*}
Then the six tensors $\rho_{ab}(\eta),\rho'_{ab}(\eta)$,
$ab\in\{11,10,01\}$, form a basis for $R_\eta$.
\end{enumerate}
\end{theorem}

\begin{proof}
We first prove (3).  Every tensor displayed in part~(3) vanishes at
$(e_s,e_s)$ for $s=0,1,2,3$, because every term involves two distinct
members of the basis $\{Y_{ab}\}$.

It remains to prove vanishing on $\Gamma_\eta$.  For compactness, write
\[
\vartheta_{p_1,\ldots,p_m}(t_1,\ldots,t_m)
=\prod_{r=1}^m\vartheta_{p_r}(t_r).
\]
Use \eqref{eq.mixed.addition.first} and
\eqref{eq.mixed.addition.second} for each of the three cyclic
permutations $(ab,ij,kl)$ listed after those equations.
Put $c=\vartheta_{ab}(0)\vartheta_{ij}(0)$ and
$\epsilon=(-1)^{k+l}$.  Replace $v$ by $-v$ in the two identities and
use the fact that $\vartheta_{11}$ is odd while the other three theta
functions are even.  Adding and subtracting the original and the
resulting identities gives
\begin{align*}
c(\vartheta_{kl,11}+\vartheta_{11,kl})(u+v,u-v)
 &=2\vartheta_{11}(u)\vartheta_{kl}(u)
      \vartheta_{ab}(v)\vartheta_{ij}(v),\\
c(\vartheta_{kl,11}-\vartheta_{11,kl})(u+v,u-v)
 &=-2\vartheta_{ab}(u)\vartheta_{ij}(u)
      \vartheta_{kl}(v)\vartheta_{11}(v),\\
c(\vartheta_{ab,ij}-\vartheta_{ij,ab})(u+v,u-v)
 &=2\epsilon\vartheta_{kl}(u)\vartheta_{11}(u)
      \vartheta_{kl}(v)\vartheta_{11}(v),\\
c(\vartheta_{ab,ij}+\vartheta_{ij,ab})(u+v,u-v)
 &=2\vartheta_{ab}(u)\vartheta_{ij}(u)
      \vartheta_{ab}(v)\vartheta_{ij}(v).
\end{align*}
Cross-multiplying the first with the third, and the second with the
fourth, gives
\begin{align}
\frac{(\vartheta_{kl,11}+\vartheta_{11,kl})(u+v,u-v)}
     {(\vartheta_{ab,ij}-\vartheta_{ij,ab})(u+v,u-v)}
&=(-1)^{k+l}
  \frac{\vartheta_{ab,ij}(v,v)}{\vartheta_{kl,11}(v,v)},
\label{eq:addition-ratio-one}\\
\frac{(\vartheta_{kl,11}-\vartheta_{11,kl})(u+v,u-v)}
     {(\vartheta_{ab,ij}+\vartheta_{ij,ab})(u+v,u-v)}
&=-\frac{\vartheta_{kl,11}(v,v)}{\vartheta_{ab,ij}(v,v)}.
\label{eq:addition-ratio-two}
\end{align}
These cross-multiplied identities remain valid when one or more of the
theta constants vanish.  Setting $(u,v)=(2z+\eta,-\eta)$ and using
$Y_{pq}(z)=\gamma_{pq}\vartheta_{pq}(2z)$ shows directly that the six
tensors in (3) vanish on $\Gamma_\eta$.  Hence they belong to $R_\eta$
for every $\eta$.

For a fixed cyclic triple, $\rho_{\bar a\bar b}$ is a linear combination of
$[Y_{11},Y_{ab}]$ and
$\{Y_{ij},Y_{kl}\}$, whereas $\rho'_{\bar a\bar b}$ is a
linear combination of
$\{Y_{11},Y_{ab}\}$ and
$[Y_{ij},Y_{kl}]$.  Neither tensor is zero: if
$\overline{\eta}\notin E_\tau[2]$, none of the four theta constants
vanishes, and if $\overline{\eta}\in E_\tau[2]$, exactly one of them
vanishes.
The two tensors are therefore independent.  The three cyclic triples
involve disjoint ordered basis tensors, so all six are independent.
Thus $\dim R_\eta\geq6$.  Since
$\dim C_\eta=2$, $R_\eta\oplus C_\eta$ is a subspace of $K_\eta$, and
$\dim K_\eta=8$, we conclude that $\dim R_\eta=6$.  The six independent
tensors therefore form a basis, proving (3).

If $\overline\eta\notin E_\tau[2]$, substitute
$X_{pq}=\vartheta_{pq}(\eta)Y_{pq}$ in
the formulas in part~(3) and divide each relation by
its nonzero common scalar.  The six resulting tensors are precisely the
six relations displayed in part~(1).  After passage to the quotient,
these are Sklyanin's defining relations under the identification of the
class $x_{ab}$ of $X_{ab}$ with the standard generator $x_{ab}$ of
$S(\eta)$, proving (1).

If $\overline\eta=0$, then $\vartheta_{11}(\eta)=0$ and the other
three theta constants are nonzero.  The six tensors in (3) reduce, up
to nonzero scalars, to the six commutators.  At a nonzero point of
$E_\tau[2]$, specializing the same formulas gives two commutators
within the two translation eigenspaces and the four anticommutators
between them.  The usual translation formulas for theta functions give
the three displayed decompositions of $V$.  Hence (2) follows.
\end{proof}

We highlight one fact that appears in the proof of
Theorem~\ref{prop:addition-formulas-relations}.

\begin{corollary}\label{cor:dim.Reta}
For all $\eta\in\CC$, $\dim R_\eta=6$.
\end{corollary}

\begin{proposition}\label{prop:order-four-relations}
Suppose $\overline\eta$ has order $4$.  After a cyclic
relabelling of $x_1,x_2,x_3$, there is an
$\alpha\in\CC-\{0,\pm1\}$ such that $A(\eta)$ has generators
$x_0,x_1,x_2,x_3$ and defining relations
\begin{align*}
 [x_0,x_1]&=\alpha\{x_2,x_3\},&
 \{x_0,x_1\}&=[x_2,x_3],\\
 x_0x_2&=x_3x_1,& x_2x_0&=-x_1x_3,\\
 x_0x_3&=-x_2x_1,& x_3x_0&=x_1x_2.
\end{align*}
Equivalently, $A(\eta)\cong S(\alpha,1,-1)$.
\end{proposition}

\begin{proof}
Theorem~\ref{prop:addition-formulas-relations}(1) identifies $A(\eta)$
with $S(\alpha_1,\alpha_2,\alpha_3)$.  When $\overline\eta$ has order
$4$, two of the structure constants are $1$ and $-1$; after a cyclic
relabelling, the parameter triple is $(\alpha,1,-1)$ with
$\alpha\notin\{0,\pm1\}$.  Smith and Stafford isolate these three
cyclically equivalent cases in \cite[\S1.9]{SmithStafford}.
Substituting $\alpha_2=1$ and $\alpha_3=-1$ into the six Sklyanin
relations gives the displayed presentation.
\end{proof}

\begin{corollary}\label{cor.alpha.parameters.and.four.torsion}
If $\overline\eta\notin E_\tau[2]$, then
$\{\alpha_{00},\alpha_{01},\alpha_{10}\}\cap\{0,1,-1\}=\varnothing$
if and only if $\overline\eta\notin E_\tau[4]$.
\end{corollary}

\begin{proof}
The hypothesis $\overline\eta\notin E_\tau[2]$ implies that the three
parameters in \eqref{eq.alpha.ab} are defined and nonzero.  For a cyclic
triple $(ab,ij,kl)$, set $u=v=\eta$ in
\eqref{eq.mixed.addition.second}.  This gives
\[
 \vartheta_{ab}(0)\vartheta_{ij}(0)^2\vartheta_{ab}(2\eta)
 =\vartheta_{ab}(\eta)^2\vartheta_{ij}(\eta)^2
  +(-1)^{k+l}\vartheta_{kl}(\eta)^2\vartheta_{11}(\eta)^2.
\]
All the theta constants on the right are nonzero.  On comparing this
identity with \eqref{eq.alpha.ab}, we see that
$\vartheta_{ab}(2\eta)=0$ if and only if $\alpha_{kl}=-1$.
Now $2\overline\eta$ is a nonzero point of $E_\tau[2]$ if and only if
one of $\vartheta_{00}(2\eta),\vartheta_{01}(2\eta),
\vartheta_{10}(2\eta)$ is zero.  Thus $\overline\eta$ has order four
if and only if one of the three parameters is $-1$.

It remains only to account for the value $1$.  If one parameter is
$1$, then the identity \eqref{eq.Sklyanin.parameter.identity} for a cyclically
ordered parameter triple $(\alpha,\beta,\gamma)$ gives, for example,
$\alpha+\beta+1+\alpha\beta=(1+\alpha)(1+\beta)=0$ when
$\gamma=1$.  Hence one of the other two parameters is $-1$.
Consequently, one of the parameters belongs to $\{1,-1\}$ if and only
if $\overline\eta$ has order four, which proves the assertion.
\end{proof}

\begin{corollary}\label{cor:Keta.decomposition}
For every $\eta\in\CC$,
$K_\eta=R_\eta\oplus C_\eta$.
\end{corollary}

\begin{proof}
Theorem~\ref{prop:addition-formulas-relations}(3) gives
$\dim R_\eta=6$,
Proposition~\ref{thm.dim.Ceta} gives $\dim C_\eta=2$, and
Proposition~\ref{prop:intersection-R-C} gives
$R_\eta\cap C_\eta=\{0\}$.  Since $R_\eta\oplus C_\eta\subseteq K_\eta$
and $\dim K_\eta=8$, equality follows.
\end{proof}

\begin{corollary}\label{cor.two.vertices.determine.Reta}
Suppose $\overline\eta\notin E_\tau[4]$.  If $i$ and $j$ are distinct
elements of $\{0,1,2,3\}$, then
\begin{equation}\label{eq.two.vertices.determine.Reta}
 R_\eta=\{r\in V^{\otimes2}\;\vert\;
 \text{$r$ vanishes on
 $\Gamma_\eta\cup\{(e_i,e_i),(e_j,e_j)\}$}\}.
\end{equation}
\end{corollary}

\begin{proof}
The right-hand side of \eqref{eq.two.vertices.determine.Reta} is the
kernel of the evaluation map
\[
 K_\eta\longrightarrow\CC^2,
 \qquad h\longmapsto
 \bigl(h(\epsilon_i,\epsilon_i),h(\epsilon_j,\epsilon_j)\bigr).
\]
Clearly, $R_\eta$ is contained in the kernel, so it remains to show
that the map $K_\eta\to\CC^2$ is surjective; i.e., that it has rank two.

Use the standard Sklyanin generators $x_0,x_1,x_2,x_3$, relabelling the
four vertices accordingly, and put
$\beta=\alpha_{01}$ and $\gamma=\alpha_{10}$.  The elements
$-x_0^2+x_1^2+x_2^2+x_3^2$ and
$x_0^2+\beta\gamma x_1^2-\gamma x_2^2+\beta x_3^2$
provide a basis for $C_\eta$.
The four coefficient columns, one for each vertex, are
\[
 \binom{-1}{\phantom{-}1},\qquad
 \binom{\phantom{-}1}{\phantom{-}\beta\gamma},\qquad
 \binom{\phantom{-}1}{-\gamma},\qquad
 \binom{\phantom{-}1}{\phantom{-}\beta}.
\]
The six determinants obtained by choosing two columns are
\[
 -(1+\beta\gamma),\quad \gamma-1,\quad -(1+\beta),
 \quad -\gamma(1+\beta),\quad \beta(1-\gamma),
 \quad \beta+\gamma.
\]
They are non-zero:
Corollary~\ref{cor.alpha.parameters.and.four.torsion} gives
$\alpha_{00},\beta,\gamma\notin\{0,\pm1\}$, so four of the six assertions
are immediate; the remaining assertions, $1+\beta\gamma\ne0$ and
$\beta+\gamma\ne0$, follow from
\eqref{eq.Sklyanin.parameter.identity}, since either equality would
force $\beta^2=1$.
Evaluation at any two distinct $(\epsilon_r,\epsilon_r)$ and
$(\epsilon_s,\epsilon_s)$ therefore provides a surjective map
$K_\eta\to\CC^2$.  The result follows.
\end{proof}

The next result will show that $C_\eta$ is a ``canonical'' complement
to $R_\eta$ in $K_\eta$ by showing that it is the only complement that
is an $H_4$-module.

Put $\widetilde T=-e(\tfrac12\tau)T$.  For $a,b\in\{0,1\}$, let
$W_{ab}$ denote the irreducible two-dimensional $H_4$-module on which
$S^2$ and $\widetilde T^2$ act as multiplication by $(-1)^a$ and
$(-1)^b$,
respectively, and put
\[
 K_{\eta,ab}:=
 \{r\in K_\eta\mid S^2r=(-1)^ar,
       \ \widetilde T^2r=(-1)^br\}.
\]

\begin{proposition}\label{prop.H4.unique.complement}
For every $\eta\in\CC$ there is a multiplicity-free decomposition
\[
 K_\eta=K_{\eta,00}\oplus K_{\eta,10}
          \oplus K_{\eta,11}\oplus K_{\eta,01},
 \qquad K_{\eta,ab}\cong W_{ab},
\]
such that
\[
 C_\eta=K_{\eta,00},\qquad
 R_\eta=K_{\eta,10}\oplus K_{\eta,11}\oplus K_{\eta,01}.
\]
Consequently, $C_\eta$ is the unique two-dimensional
$H_4$-submodule of $K_\eta$ having zero intersection with $R_\eta$.
\end{proposition}

\begin{proof}
The formulas in Lemma~\ref{lem.H4.action.on.V} show that the displayed
spaces are the simultaneous eigenspaces of $S^2$ and
$\widetilde T^2$ on
$K_\eta$.  Proposition~\ref{thm.dim.Ceta} and its explicit formula
$C_\eta=\theta(z_1-z_2\pm\eta)W(z_1+z_2)$ give
$C_\eta=K_{\eta,00}\cong W_{00}$.  For the three nonzero pairs
$ab$, the elements $\rho_{ab}$ and $\rho'_{ab}$ in
Theorem~\ref{thm:R-eta-basis}\textup{(3)} are a basis for
$K_{\eta,ab}$.  Proposition
\ref{prop.dim.Keta} now forces all four eigenspaces to have dimension
two and gives both displayed decompositions.

The four summands are pairwise nonisomorphic irreducible
$H_4$-modules.  Hence an $H_4$-submodule of dimension two having zero
intersection with $R_\eta$ must be the remaining irreducible summand
$K_{\eta,00}=C_\eta$.
\end{proof}

The isomorphism $A(\eta)\to S(\eta)$ in
Theorem~\ref{prop:addition-formulas-relations} sends the class
$x_{ab}$ of $X_{ab}$ to the standard generator $x_{ab}$ of $S(\eta)$.
Smith--Stafford's generators $x_0,x_1,x_2,x_3$ are our
$x_{11},x_{00},x_{01},x_{10}$, respectively.  Hence, by
\eqref{eq.Sklyanin.generators.to.x}, the correspondence between the
generators of $A(\eta)$ and Sklyanin's generators is
\[
 x_{11}\longmapsto S_0,\qquad
 x_{10}\longmapsto \ii S_1,\qquad
 x_{00}\longmapsto-\ii S_2,\qquad
 x_{01}\longmapsto \ii S_3.
\]
Rains--Ruijsenaars also use Sklyanin's generators
$S_0,S_1,S_2,S_3$ in \cite[(B.31)--(B.34)]{RR}.

\subsection{The Rains--Ruijsenaars algebra
\texorpdfstring{$\mathcal R(\eta)$ and a proof that
$A(\eta)=\mathcal R(\eta)$}
{R(eta) and a proof that A(eta)=R(eta)}}
\label{sect.RR-algebra}

In this section we use the notation $\nu:=\frac12\eta$.

We copy Rains--Ruijsenaars~\cite[(B.2)--(B.5)]{RR} and define
\begin{align}
f(a,b,c;z)
&:=\theta(z+a)\theta(z+b)\theta(z+c)\theta(z-a-b-c),
\label{eq.RR.f.definition}\\
F_\eta(a,b,c_1,c_2;z_1,z_2)
&:=\begin{aligned}[t]
&f(a-\nu,b-\nu,c_1+\nu;z_1)
 f(a+\nu,b+\nu,c_2-\nu;z_2)\\
&-f(a-\nu,b-\nu,c_2+\nu;z_1)
 f(a+\nu,b+\nu,c_1-\nu;z_2),
\end{aligned}
\label{eq.RR.F.definition}\\
R_\eta^{\mathrm{RR}}
&:=\begin{aligned}[t]
&\text{the linear span of the functions in
\eqref{eq.RR.F.definition}}\\[-2pt]
&\text{as $a,b,c_1,c_2$ vary},
\end{aligned}\notag\\
\mathcal R(\eta)&:=\frac{TV}{(R_\eta^{\mathrm{RR}})}.\notag
\end{align}

Each
$f(a,b,c)=f(a,b,c;z)$ belongs to $V=V(z)$ and each
$F_\eta(a,b,c_1,c_2;z_1,z_2)$ belongs to
$V^{\otimes2}=V(z_1)\otimes V(z_2)$.

Relations of the form $u_1\otimes v_2-u_2\otimes v_1$, with
$u_1,u_2,v_1,v_2\in V$, play a crucial role in showing that the
isomorphism classes of line modules for the Sklyanin algebra are in
natural bijection with the secant lines to $E_\tau$ in $\PP V^*$;
see \cite[Section~4]{LS}.  Corollary~\ref{prop:RR.relations.and.LS.secants}
shows that the functions
$F_\eta(a,b,c_1,c_2)$ are essentially the same as the
Levasseur--Smith relations.  The next result is the first step towards
that conclusion.

\begin{lemma}\label{lem:RR.functions.vanish.on.graph}
Let $a,b,c_1,c_2\in\CC$.  The following statements hold.
\begin{enumerate}
\item $F_\eta(a,b,c_1,c_2)
=u_{c_1}\otimes v_{c_2}-u_{c_2}\otimes v_{c_1}$, where
\[
 u_c(z)=f(a-\nu,b-\nu,c+\nu;z),
 \qquad
 v_c(z)=f(a+\nu,b+\nu,c-\nu;z).
\]

\item $F_\eta(a,b,c_1,c_2)$ vanishes on $\Gamma_\eta$.

\item The following conditions are equivalent:
\begin{enumerate}
\item $F_\eta(a,b,c_1,c_2)\ne0$;
\item $c_1-c_2\notin\Lambda_\tau$ and
$a+b+c_1+c_2\notin\Lambda_\tau$;
\item $u_{c_1}$ and $u_{c_2}$ are linearly independent;
\item $v_{c_1}$ and $v_{c_2}$ are linearly independent.
\end{enumerate}
When these conditions hold,
 \[
 \{u_{c_1}=u_{c_2}=0\}
 =\overline{-\overline a+\overline\nu,\;-\overline b+\overline\nu}
 \quad\text{and}\quad
 \{v_{c_1}=v_{c_2}=0\}
 =\overline{-\overline a-\overline\nu,\;-\overline b-\overline\nu}.
 \]

\item The first of the secant lines in \textup{(3)} is obtained from the
second by applying $t_{\overline\eta}$.

\item Under the equivalent conditions in \textup{(3)},
$F_\eta(a,b,c_1,c_2)$ is a tensor of rank two.

\item Replacing $c_1,c_2$ by another pair satisfying the equivalent
conditions in \textup{(3)} changes
$F_\eta(a,b,c_1,c_2)$ by a nonzero scalar multiple.
\end{enumerate}
\end{lemma}

\begin{proof}
Part~\textup{(1)} follows directly from the definitions of $F_\eta$,
$u_c$, and $v_c$.

For $z\in\CC$, the factors in $u_c(z)$ and $v_c(z+\eta)$ that
involve $c$ are the same, namely
\[
 \theta(z+c+\nu)\theta(z-a-b-c+\nu).
\]
Thus, as functions of $c$, $u_c(z)$ and $v_c(z+\eta)$ are
proportional, with proportionality factor independent of $c$.
The determinant of the matrix
\[
 \begin{pmatrix}
 u_{c_1}(z)&u_{c_2}(z)\\
 v_{c_1}(z+\eta)&v_{c_2}(z+\eta)
 \end{pmatrix}
\]
is therefore zero, i.e.,
$F_\eta(a,b,c_1,c_2;z,z+\eta)=0$.
This proves~\textup{(2)}.

The equivalence of \textup{(3a)} and \textup{(3b)} follows at once from
the definition of $F_\eta(a,b,c_1,c_2)$.
The functions $u_{c_1}$ and $u_{c_2}$ are divisible by
$\theta(z+a-\nu)\theta(z+b-\nu)$.  The other two zeros of $u_c$ are
$-\overline c-\overline\nu$ and
$\overline a+\overline b+\overline c-\overline\nu$.  Hence
$u_{c_1}$ and $u_{c_2}$ are linearly dependent if and only if their
remaining zero divisors agree, i.e., if and only if
$c_1-c_2\in\Lambda_\tau$ or
$a+b+c_1+c_2\in\Lambda_\tau$.  Thus the first three
conditions in~\textup{(3)} are equivalent.  When they hold, the common
zero locus of $u_{c_1}$ and $u_{c_2}$ is as claimed.  Fix such $c_1$
and $c_2$.

Let $(\alpha_1,\alpha_2)\in\CC^2$.  Because the
determinant vanishes,
$\alpha_1u_{c_1}(z)+\alpha_2u_{c_2}(z)=0$ for all $z$ if and only if
$\alpha_1v_{c_1}(z+\eta)+\alpha_2v_{c_2}(z+\eta)=0$ for all $z$; but the latter
equality holds if and only if
$\alpha_1v_{c_1}(z)+\alpha_2v_{c_2}(z)=0$ for all $z$.  Hence there is a linear
isomorphism $\CC u_{c_1}\oplus\CC u_{c_2}\to
\CC v_{c_1}\oplus\CC v_{c_2}$ that sends $u_{c_i}$ to $v_{c_i}$.
In particular,
$v_{c_1},v_{c_2}$ are linearly independent.  Since they are divisible by
$\theta(z+a+\nu)\theta(z+b+\nu)$, their common zero locus is the
second secant line in~\textup{(3)}.  Since $2\overline\nu=\overline\eta$,
applying $t_{\overline\eta}$ to the second secant line gives the first,
which proves~\textup{(4)}.  The same argument with $u$ and $v$
interchanged shows that linear independence of $v_{c_1},v_{c_2}$
implies linear independence of $u_{c_1},u_{c_2}$.  This completes the
proof of all the equivalences in~\textup{(3)}.

For such $c_1,c_2$, both $\{u_{c_1},u_{c_2}\}$ and
$\{v_{c_1},v_{c_2}\}$ are linearly independent.  It follows
from~\textup{(1)} that
$F_\eta(a,b,c_1,c_2)$ is a nonzero tensor of rank two, proving~\textup{(5)}.
If $d_1,d_2$ is another pair for which $u_{d_1},u_{d_2}$ are linearly
independent, let $M$ be the change-of-basis matrix from the basis
$\{u_{c_1},u_{c_2}\}$ to the basis $\{u_{d_1},u_{d_2}\}$.  The
isomorphism above shows that $M$ is also the change-of-basis matrix
from the basis $\{v_{c_1},v_{c_2}\}$ to the basis
$\{v_{d_1},v_{d_2}\}$.  Hence
\[
 F_\eta(a,b,d_1,d_2)=\det(M)F_\eta(a,b,c_1,c_2).
\]
Since $\det(M)\ne0$, this proves~\textup{(6)}.
\end{proof}

The fact that $F_\eta(a,b,c_1,c_2)$ vanishes on $\Gamma_\eta$ is
equivalent to the fact that $\theta(z_1-z_2+\eta)$ divides
$F_\eta(a,b,c_1,c_2;z_1,z_2)$.  This divisibility is an important
part of the proof of \cite[Lemma~B.1]{RR}, where it is established
using Weierstrass's three-term identity
\eqref{eq.Weierstrass.three-term}.

The next result describes the zero locus of
$F_\eta(a,b,c_1,c_2)$ in $E_\tau\times E_\tau$; it will be used in
the proof of Lemma~\ref{prop:six.independent.RR.relations}.

\begin{lemma}\label{lem:RR.secant.zero.test}
Let $a,b\in\CC$, and choose $c_1,c_2\in\CC$ so that
$\{u_{c_1},u_{c_2}\}$ is linearly independent.  The zero locus of
$F_\eta(a,b,c_1,c_2)$ is described, with multiplicities, by the
following equality of effective divisors on $E_\tau\times E_\tau$:
\begin{equation}\label{eq.RR.secant.zero.test}
\begin{aligned}
 \operatorname{div}\bigl(F_\eta(a,b,c_1,c_2)\bigr)
 ={}&\bigl(\{-\overline a+\overline\nu\}\times E_\tau\bigr)
    +\bigl(\{-\overline b+\overline\nu\}\times E_\tau\bigr)\\
 &+\bigl(E_\tau\times\{-\overline a-\overline\nu\}\bigr)
    +\bigl(E_\tau\times\{-\overline b-\overline\nu\}\bigr)\\
 &+\Gamma_\eta
    +\bigl(\{(x,y)\in E_\tau\times E_\tau
       \mid x+y=\overline a+\overline b\}\bigr).
\end{aligned}
\end{equation}
\end{lemma}

\begin{proof}
The formulas for $u_c$ and $v_c$ show that the first four components
on the right-hand side of \eqref{eq.RR.secant.zero.test} occur with
the indicated multiplicities.  Thus, if $\overline a=\overline b$,
the corresponding vertical and horizontal components each occur with
multiplicity two.  Lemma~\ref{lem:RR.functions.vanish.on.graph} shows
that $F_\eta(a,b,c_1,c_2)$ also vanishes on $\Gamma_\eta$.

Fix a general $x\in\CC$ such that
$\overline x\notin\{-\overline a+\overline\nu,
                    -\overline b+\overline\nu\}$.
The function $F_\eta(a,b,c_1,c_2;x,-)$ belongs to $V$.  By
Lemma~\ref{lem:RR.functions.vanish.on.graph}, it vanishes at
$\overline x+\overline\eta$,
$-\overline a-\overline\nu$, and
$-\overline b-\overline\nu$.  It is nonzero because
$\{v_{c_1},v_{c_2}\}$ is linearly independent and
$(u_{c_1}(x),u_{c_2}(x))\ne(0,0)$.  Its four zeros, counted with
multiplicity, sum to $0$ in $E_\tau$,
so it also vanishes at
$\overline a+\overline b-\overline x$.
Its remaining zero is therefore the intersection of
$\{\overline x\}\times E_\tau$ with the last term on the right-hand
side of \eqref{eq.RR.secant.zero.test}.  The analogous argument after fixing a
general value of the second variable accounts for the vertical
components and proves the equality of divisors in
\eqref{eq.RR.secant.zero.test}.
\end{proof}

We now use secant lines through certain $2$-torsion and $4$-torsion
points to obtain six independent
Rains--Ruijsenaars relations without using the product formula (B.44).

\begin{lemma}\label{prop:six.independent.RR.relations}
For every $\eta\in\CC$, there are six linearly independent elements of
$R_\eta^{\mathrm{RR}}$ associated to six secant or tangent lines.
More precisely, choose $t_0,t_1,t_2\in E_\tau$ as follows.
\begin{enumerate}
\item If $\overline\eta$ does not have exact order four, choose
$p_1,p_2$ such that
$E_\tau[2]=\{0,p_1,p_2,p_1+p_2\}$ and put
$(t_0,t_1,t_2)=(0,p_1,p_2)$.
\item If $\overline\eta$ has exact order four, choose
$\xi\in E_\tau[4]$ such that $\{\overline\eta,\xi\}$ is a basis for
$E_\tau[4]$ and put $(t_0,t_1,t_2)=(0,\xi,2\xi)$.
\end{enumerate}
Choose lifts $\widetilde t_i\in\CC$ of $t_i$, and put
$a_i=-\widetilde t_i-\nu$.  Then the six relations
\[
 r_{t_i,t_j}:=
 F_\eta(a_i,a_j,c_1^{ij},c_2^{ij})
 \qquad(0\leq i\leq j\leq2),
\]
where, for each $(i,j)$, $c_1^{ij},c_2^{ij}$ are chosen as in the
statement of Lemma~\ref{lem:RR.secant.zero.test}, are linearly independent.  They are
associated to the three tangent lines $T_{t_i}E_\tau$ and the three
secant lines $\overline{t_it_j}$, $i<j$.
\end{lemma}

\begin{proof}
Suppose that
$\sum_{0\leq i\leq j\leq2}\lambda_{ij}r_{t_i,t_j}=0$.
For $0\leq i\leq j\leq2$, Lemma~\ref{lem:RR.functions.vanish.on.graph}
shows that the common zero locus of the two functions in the left
tensor factors of $r_{t_i,t_j}$ is
$\overline{t_i+\overline\eta,\,t_j+\overline\eta}$, whereas the
common zero locus of the two functions in its right tensor factors is
$\overline{t_i,t_j}$.  We call these two lines the left and right base
secants of $r_{t_i,t_j}$, respectively.

We first show that the coefficients of the three tangent relations
vanish.  For each $i\in\{0,1,2\}$, let $k<l$ be the two elements of
$\{0,1,2\}-\{i\}$.  Thus
\[
 (k,l)=(1,2),\ (0,2),\ (0,1)
 \qquad\text{for }i=0,1,2,\text{ respectively}.
\]
Evaluate the linear relation at $(t_k+\eta,t_l)$.  Each of its five
terms other than $\lambda_{ii}r_{t_i,t_i}$ vanishes: for each such
term, either $t_k+\eta$ belongs to its left base secant or $t_l$
belongs to its right base secant.  By
\eqref{eq.RR.secant.zero.test}, the value
$r_{t_i,t_i}(t_k+\eta,t_l)$ is nonzero provided that
$t_k-t_l+2\eta\ne0$ and $t_k+t_l+2t_i+2\eta\ne0$ in $E_\tau$.

In the first case of the lemma, the two quantities are equal for each
$i$, and their common values for $i=0,1,2$ are, respectively,
$p_1+p_2+2\overline\eta$, $p_2+2\overline\eta$, and
$p_1+2\overline\eta$.
They are nonzero because $2\overline\eta$ is not a nonzero
two-torsion point.  In the exact-order-four case, the two quantities
are displayed in the corresponding row of
\[
\begin{array}{c|cc}
i&t_k-t_l+2\overline\eta&t_k+t_l+2t_i+2\overline\eta\\ \hline
0&2\overline\eta-\xi&2\overline\eta-\xi\\
1&2\overline\eta-2\xi&2\overline\eta\\
2&2\overline\eta-\xi&2\overline\eta+\xi .
\end{array}
\]
None is zero because $\{\overline\eta,\xi\}$ is a basis for $E_\tau[4]$.
Thus $r_{t_i,t_i}(t_k+\eta,t_l)\ne0$ in both cases of the lemma.
The evaluated relation therefore gives $\lambda_{ii}=0$.  Since this
holds for each $i$, $\lambda_{00}=\lambda_{11}=\lambda_{22}=0$.

It remains to consider the three secant relations.  Fix $i<j$, and let
$k$ be the unique element of $\{0,1,2\}-\{i,j\}$.  Apply evaluation at
$t_k+\eta$ to the first tensor factor of the relation.  Among the
three terms that remain, the two other than
$\lambda_{ij}r_{t_i,t_j}$ become zero because their left base secants
contain $t_k+\eta$.  The image of $r_{t_i,t_j}$ is nonzero: this
tensor has rank two and $t_k+\eta$ does not belong to its left base
secant.  Hence
$\lambda_{ij}=0$.  This proves the independence of all six relations.
\end{proof}

\begin{corollary}\label{cor:dim.RR.at.least.six}
For every $\eta\in\CC$, $\dim R_\eta^{\mathrm{RR}}\geq6$.
\end{corollary}

Appendix~\ref{app:uniform.RR.basis} gives a different, parameter-uniform
choice of six independent relations.
\begin{theorem}\label{thm:Reta.equals.RR}
For all $\eta\in\CC$, $R_\eta^{\mathrm{RR}}=R_\eta$ and
\[
 \mathcal R(\eta)=A(\eta).
\]
\end{theorem}

\begin{proof}
Since $\dim R_\eta=6$ and $\dim R_\eta^{\mathrm{RR}}\geq6$, it
suffices to show that $R_\eta^{\mathrm{RR}}\subseteq R_\eta$.
By Lemma~\ref{lem:RR.functions.vanish.on.graph}, every
$F_\eta(a,b,c_1,c_2)$ vanishes on $\Gamma_\eta$.
To prove the inclusion, it therefore suffices to show that every
$F_\eta(a,b,c_1,c_2)$ vanishes on the four vertex pairs.

We start the proof with five bookkeeping items.
\begin{enumerate}
\item To avoid confusing the indices of theta functions with the complex
parameters $a,b,c$ in the notation of Rains--Ruijsenaars, we use $i,j$
for the indices in $\vartheta_{ij}$.  The elements $\epsilon_{ij}$ and
the points $e_{ij}$ are those defined in \S2.1; by
Corollary~\ref{cor.eij.are.vertices}, the $e_{ij}$ are the vertices of
the four singular quadrics containing $E_\tau$.

\item We write
$F_\eta(a,b,c_1,c_2)=u_{c_1}\otimes v_{c_2}
-u_{c_2}\otimes v_{c_1}$, where $u_c$ and $v_c$ are the functions in
Lemma~\ref{lem:RR.functions.vanish.on.graph}\textup{(1)}.
The function $F_\eta(a,b,c_1,c_2)$ vanishes at
$(e_{ij},e_{ij})$ if and only if
\[
 \epsilon_{ij}(u_{c_1})\epsilon_{ij}(v_{c_2})
 -\epsilon_{ij}(u_{c_2})\epsilon_{ij}(v_{c_1})=0.
\]

\item Recall that $\theta=-\vartheta_{11}$.

\item We will apply Jacobi's five-term four-theta identity
\eqref{eq.Jacobi.five-term.four-theta} to
\begin{align*}
 f(a,b,c;z)
 &=\theta(z+a)\theta(z+b)\theta(z+c)
   \theta(z-a-b-c)\\
 &=\vartheta_{11}(z+a)\vartheta_{11}(z+b)
   \vartheta_{11}(z+c)\vartheta_{11}(z-a-b-c).
\end{align*}
If $(x_1,x_2,x_3,x_4)=(z+a,z+b,z+c,z-a-b-c)$, then
\[
 (y_1,y_2,y_3,y_4)=(2z,a+b,a+c,-b-c).
\]

\item Define $s_{11}=s_{00}=1$ and $s_{10}=s_{01}=-1$.
By \eqref{eq.Jacobi.five-term.four-theta},
\[
 2\epsilon_{ij}\bigl(f(a,b,c;z)\bigr)
 =s_{ij}\vartheta_{ij}(a+b)\vartheta_{ij}(a+c)
  \vartheta_{ij}(-b-c).
\]
\end{enumerate}

We now apply the last formula to $u_c$ and $v_c$.  Since
$2\nu=\eta$, it gives
\begin{align*}
 2\epsilon_{ij}(u_c)
 &=s_{ij}\vartheta_{ij}(a+b-\eta)
   \vartheta_{ij}(a+c)\vartheta_{ij}(-b-c),\\
 2\epsilon_{ij}(v_c)
 &=s_{ij}\vartheta_{ij}(a+b+\eta)
   \vartheta_{ij}(a+c)\vartheta_{ij}(-b-c).
\end{align*}
If we define $h_{ij}(c):=\tfrac12s_{ij}
\vartheta_{ij}(a+c)\vartheta_{ij}(-b-c)$, then
\[
 \epsilon_{ij}(u_c)=\vartheta_{ij}(a+b-\eta)h_{ij}(c),
 \qquad
 \epsilon_{ij}(v_c)=\vartheta_{ij}(a+b+\eta)h_{ij}(c)
\]
so the determinant of the matrix 
\[
 \begin{pmatrix}
  \epsilon_{ij}(u_{c_1})&\epsilon_{ij}(u_{c_2})\\
  \epsilon_{ij}(v_{c_1})&\epsilon_{ij}(v_{c_2})
 \end{pmatrix}
 =
 \begin{pmatrix}
  \vartheta_{ij}(a+b-\eta)\\
  \vartheta_{ij}(a+b+\eta)
 \end{pmatrix}
 \begin{pmatrix}h_{ij}(c_1)&h_{ij}(c_2)\end{pmatrix}
\]
is zero.  Hence, by part~\textup{(2)} above,
$F_\eta(a,b,c_1,c_2;e_{ij},e_{ij})=0$ for all
$i,j\in\{0,1\}$.
Since the functions $F_\eta(a,b,c_1,c_2)$ vanish at all four vertex pairs,
$R_\eta^{\mathrm{RR}}\subseteq R_\eta$.  Hence
$R_\eta^{\mathrm{RR}}=R_\eta$, and the result follows from the
definitions of $\mathcal R(\eta)$ and $A(\eta)$.
\end{proof}

\begin{corollary}\label{prop:RR.relations.and.LS.secants}
Suppose that $\overline\eta\notin E_\tau[4]$.  Up to nonzero scalar
multiples, the nonzero Rains--Ruijsenaars functions
$F_\eta(a,b,c_1,c_2)$ are precisely the relations associated by
Levasseur--Smith to the secant lines of $E_\tau$.
More explicitly, $F_\eta(a,b,c_1,c_2)$ is associated to the line
$\overline{pq}$, where $p=-\overline a-\overline\nu$ and
$q=-\overline b-\overline\nu$.
\end{corollary}

\begin{proof}
Suppose $F_\eta(a,b,c_1,c_2)\ne0$.  Let $p,q\in E_\tau$ be the points
$-\overline a-\overline\nu$ and $-\overline b-\overline\nu$.
Lemma~\ref{lem:RR.functions.vanish.on.graph}\textup{(3)} shows
that $\{v_{c_1},v_{c_2}\}$ is a basis for the space of linear forms
vanishing on $\overline{pq}$.  Theorem~\ref{thm:Reta.equals.RR} shows
that $F_\eta(a,b,c_1,c_2)$ belongs to $R_\eta$.  Hence, by the
uniqueness result in \cite[Lem.~4.3 and Prop.~4.4]{LS},
$F_\eta(a,b,c_1,c_2)$ is a nonzero scalar multiple of ``the''
Levasseur--Smith relation associated to $\overline{pq}$.

Conversely, let $\overline{pq}$ be a secant line and let
$u_1\otimes v_2-u_2\otimes v_1$ be ``the'' associated
Levasseur--Smith relation.
Choose $a,b\in\CC$ such that
$\overline a=-p-\overline\nu$ and
$\overline b=-q-\overline\nu$, and choose $c_1,c_2$ so that
$\{v_{c_1},v_{c_2}\}$ is linearly independent, and hence a basis for
the linear forms vanishing on $\overline{pq}$.  By the first part
of the proof, $F_\eta(a,b,c_1,c_2)$ is a nonzero scalar multiple of
$u_1\otimes v_2-u_2\otimes v_1$.
\end{proof}

\begin{corollary}
For every $\eta\in\CC$, the six secant relations in
Lemma~\ref{prop:six.independent.RR.relations} form a basis for
$R_\eta$.  Thus $R_\eta$ has a basis consisting of Rains--Ruijsenaars
functions of the form $u_1\otimes v_2-u_2\otimes v_1$.
\end{corollary}

\begin{remark}
There is another route to Theorem~\ref{thm:Reta.equals.RR}.  After the
generators and parameters are translated into our notation,
Rains--Ruijsenaars prove that
$\mathcal R(\eta)\cong S(\eta)$ when
$\overline\eta\notin E_\tau[2]$.  Combining their result with
Theorem~\ref{prop:addition-formulas-relations} identifies their relation
space with $R_\eta$
on that locus.  Since both six-dimensional relation spaces depend
holomorphically on $\eta$ as points of
$\operatorname{Gr}(6,V^{\otimes2})$, the equality extends to every
$\eta$.  Our proof above is a useful complement to that argument: it
proves the equality directly at all parameters and exhibits a basis
arising from secant lines, without first passing through
$S(\eta)$.
\end{remark}

\subsection{Feigin and Odesskii's algebras
\texorpdfstring{$Q_{4,1}(\eta)$ and the proof
$A(\eta)\cong Q_{4,1}(-\tfrac12\eta)$}
{Q41(eta) and the proof A(eta) is isomorphic to Q41(-eta/2)}}

The functions $\theta_0,\theta_1,\theta_2,\theta_3$ used in
this subsection are the functions defined in \S2.3 by specializing
\cite[Prop.~2.6]{CKS1} to $n=4$,
not the numbered theta functions defined in
\eqref{eq.numbered.theta.conventions}.

When $\overline{\eta}\notin E_\tau[4]$,
$Q_{4,1}(\eta)$ is the free algebra
$\CC\langle x_0,x_1,x_2,x_3\rangle$ modulo the relations
\begin{equation}
\label{eq.Q41.relations}
 r_{ij}(\eta)=
 \sum_{r\in\ZZ_4}
 \frac{1}
      {\theta_{j-i-r}(-\eta)\theta_r(\eta)}
 x_{j-r}\otimes x_{i+r},
 \qquad i,j\in\ZZ_4,\quad i\ne j.
\end{equation}
Let
\[
 \rel_{4,1}(\eta)
   =\Span\{r_{ij}(\eta)\mid i,j\in\ZZ_4,\ i\ne j\}
   .
\]
The definition extends to all $\eta$ as in \cite[Section~3.3]{CKS1}.

All subscripts in these formulas are read modulo four.

\begin{lemma}
\label{lem.six.CKS.relations}
For all $\eta\in\CC$, $\dim\rel_{4,1}(\eta)=6$.
\end{lemma}

\begin{proof}
If $\overline{\eta}\notin E_\tau[4]$, Proposition~3.4 of \cite{CKS1}
exhibits an explicit linear isomorphism
\[
 \operatorname{Alt}^2\Theta_4(\tau)\longrightarrow\rel_{4,1}(\eta),
\]
so
$\dim\rel_{4,1}(\eta)=\binom42=6$.  If
$\overline{\eta}\in E_\tau[4]$, the same dimension
statement is \cite[Cor.~5.3]{CKS1}.
\end{proof}

We identify the generating subspace of
$Q_{4,1}(\eta)$ with $\Theta_4(\tau)$ by setting
$x_\alpha=\theta_\alpha$ for $\alpha=0,1,2,3$
(cf.~\cite[Section~5.2]{CKS2}).

Thus $A(\eta)$ and $Q_{4,1}(\eta)$ are generated by different
spaces of order-four theta functions, namely $V_0$ and $V_{1/2}$ in
the notation of \S2.1.  The isomorphism
$A(\eta)\to Q_{4,1}(-\tfrac12\eta)$, which we prove in
Theorem~\ref{thm.Q41.Sklyanin}, is induced by the translation
isomorphism $\Psi=\Psi_{1/8}:V \to \Theta_4(\tau)$ in Lemma~\ref{prop.translation.V.Vprime}.

\begin{proposition}
\label{prop.Q41.relations.in.R}
For all $\eta\in\CC$, $(\Psi\otimes\Psi)(R_{-2\eta})=\rel_{4,1}(\eta)$.
\end{proposition}

\begin{proof}
Since ${\rm dim}(\rel_{4,1}(\eta)) = {\rm dim}(R_{-2\eta})$,
it suffices to show that
$(\Psi^{-1}\otimes\Psi^{-1})(\rel_{4,1}(\eta))
\subseteq R_{-2\eta}$.
By \cite[Lem.~3.6]{CKS1}, $\rel_{n,1}(\eta)$ vanishes on
$\{(z,z-(n-2)\overline\eta)\mid z\in E_\tau\}$.
Hence $\rel_{4,1}(\eta)$ vanishes on  
\begin{equation}
\label{eq.CKS.graph}
 \Gamma'_{-2\overline\eta}
   :=\{(z,z-2\overline\eta)\mid z\in E_\tau\}
   \subset\PP\Theta_4(\tau)^*\times\PP\Theta_4(\tau)^*.
\end{equation}

First suppose that $\overline{\eta}\notin E_\tau[4]$.  We will now show
that every $r_{ij}(\eta)$ in \eqref{eq.Q41.relations} vanishes at every
$(e_{\alpha,3/2},e_{\alpha,3/2})$.  At each of the vertex points
$e_{\alpha,3/2}$ in
\eqref{eq.vertices.CKS.Prop.2.7.basis}, precisely two coordinates of
the same parity are nonzero.  If $j-i$ is odd, then the indices
$j-r$ and $i+r$ have opposite parity.  Hence
$x_{j-r}\otimes x_{i+r}$ vanishes at each
$(e_{\alpha,3/2},e_{\alpha,3/2})$.
It remains to consider $j-i=2$ in $\ZZ_4$.  In the sum defining
$r_{ij}(\eta)$, group together, for each $r\in\ZZ_4$, the summand
whose tensor factor is $x_{j-r}\otimes x_{i+r}$ and the summand indexed
by $r+2$, whose tensor factor is
$x_{j-r-2}\otimes x_{i+r+2}$.  Since
$\theta_a(-z)=-e\!\left(-4z+\tfrac14a\right)\theta_{-a}(z)$
\cite[Prop.~2.6(5)]{CKS1}, we obtain
\[
 \frac{1}{\theta_{2-r}(-\eta)\theta_r(\eta)}
 =-
 \frac{1}{\theta_{-r}(-\eta)\theta_{r+2}(\eta)}.
\]
Hence the scalar coefficients of these two
summands are negatives of one another.  At any one of the four points
$e_{\alpha,3/2}$, translating both coordinate indices by $2$ changes the two
nonzero coordinate values by the same sign.  Consequently the two
tensors have the same value at $(e_{\alpha,3/2},e_{\alpha,3/2})$.  Their
contributions to $r_{ij}(\eta)(e_{\alpha,3/2},e_{\alpha,3/2})$ therefore cancel.
Hence, for every $i\ne j$ and every $\alpha$,
$r_{ij}(\eta)(e_{\alpha,3/2},e_{\alpha,3/2})=0$.

We now assume that $\overline{\eta}\in E_\tau[4]$.  The holomorphically
normalized relations
$h_{ij}$ constructed in \cite[Section~3.2.3]{CKS1} extend across these
parameters by \cite[Lem.~3.5]{CKS1} and span
$\rel_{4,1}(\eta)$ for every $\eta$.  Since their span has the constant
dimension six by Lemma~\ref{lem.six.CKS.relations}, the map
$\overline\eta\mapsto\rel_{4,1}(\eta)$ is a holomorphic map
$E_\tau\to\operatorname{Gr}(6,\Theta_4(\tau)^{\otimes2})$.  The subspace of
$\Theta_4(\tau)^{\otimes2}$ consisting of tensors that vanish at the four
pairs $(e_{\alpha,3/2},e_{\alpha,3/2})$ is independent of $\eta$ and
closed.  For every
$\eta\notin\frac14\Lambda_\tau$, the calculation above shows that
$\rel_{4,1}(\eta)$ is contained in this subspace.  By continuity, the same
containment holds when $\eta\in\frac14\Lambda_\tau$.  Thus every relation in
$\rel_{4,1}(\eta)$ vanishes at the four vertex pairs for every $\eta$.

By equation~\eqref{eq.vertices.transport},
$(\Psi^*,\Psi^*)(e_{\alpha,3/2},e_{\alpha,3/2})=(e_\alpha,e_\alpha)$.
Moreover,
$\Psi^*$ translates both projective factors by the same point, so
it carries $\Gamma'_{-2\eta}$ to $\Gamma_{-2\eta}$.  The transported
relation space therefore vanishes both on $\Gamma_{-2\eta}$ and at
those four pairs.  These are precisely the vanishing conditions that
define $R_{-2\eta}$, and hence
$(\Psi^{-1}\otimes\Psi^{-1})(\rel_{4,1}(\eta))\subseteq R_{-2\eta}$.
Applying $\Psi\otimes\Psi$ gives the required inclusion.  Since the two
spaces have the same dimension, it is an equality.
\end{proof}

\begin{theorem}
\label{thm.Q41.Sklyanin}
For all $\eta\in\CC$, the isomorphism
$\Psi:V\to\Theta_4(\tau)$ extends to an isomorphism
\[
 A(\eta)\cong Q_{4,1}(-\tfrac12\eta).
\]
\end{theorem}

\begin{proof}
Apply Proposition~\ref{prop.Q41.relations.in.R} with parameter
$-\tfrac12\eta$.
\end{proof}

\begin{remark}
The isomorphism in Theorem~\ref{thm.Q41.Sklyanin}, together with
Proposition~\ref{prop:A-under-elliptic-automorphism}(2), yields
\[
 A(\eta)\cong Q_{4,1}(-\tfrac12\eta)
 \cong Q_{4,1}(\tfrac12\eta)^{\mathrm{op}}.
\]
Indeed, applying the theorem with parameter $-\eta$ gives
$A(-\eta)\cong Q_{4,1}(\tfrac12\eta)$, and the proposition identifies
$A(-\eta)$ with $A(\eta)^{\mathrm{op}}$.
\end{remark}

\begin{corollary}\label{cor.Q41.two-torsion.polynomial}
If $\overline{\eta_1-\eta_2}\in E_\tau[2]$, then
$Q_{4,1}(\eta_1)=Q_{4,1}(\eta_2)$.  In particular, if
$\overline{\eta}\in E_\tau[2]$, then $Q_{4,1}(\eta)=S\Theta_4(\tau)$.
\end{corollary}

\begin{proof}
The hypothesis implies
$\overline{-2\eta_1}=\overline{-2\eta_2}$, so
$R_{-2\eta_1}=R_{-2\eta_2}$.  Theorem~\ref{thm.Q41.Sklyanin} therefore
gives the asserted equality.  For the particular case, take
$\eta_1=\eta$ and $\eta_2=0$.  Then
$R_{-2\eta}=R_0=\bigwedge^2V$, whence
$Q_{4,1}(\eta)=TV/(\bigwedge^2\Theta_4(\tau))=S\Theta_4(\tau)$.
\end{proof}
 
\begin{remark}
There are some private calculations, and some published remarks, to
the effect that the two families are the same; i.e., every $S(\eta)$
is isomorphic to some $Q_{4,1}(\eta')$ and, conversely, if
$\overline{\eta'}\notin E_\tau[4]$, then $Q_{4,1}(\eta')$ is
isomorphic to some $S(\eta)$.  Toby Stafford sent
me a proof of this in a letter dated October 1, 1990, and we alluded
to this fact in \cite[\S5.7]{SmithStafford}.  Having forgotten
that letter, Alex Chirvasitu, Kenneth Chan, and I gave another proof in
2016.  In June 2023, Inomjon Mizomov sent me another proof, which has
now appeared in \cite{Mizomov}.  All these proofs show that $A(\eta)$
is isomorphic to $Q_{4,1}(\eta')$ for some $\eta'$, but $\eta'$ is not
identified.  All these proofs are calculations and do not explain why
such isomorphisms exist.  Tate and Van den Bergh remarked that the two
classes of algebras are the same \cite[p.~620]{TateVdB}.   

The proof above provides a conceptual understanding of the
isomorphism.  The generating spaces of theta functions for $A(\eta)$
and $Q_{4,1}(\eta')$ are not the same: the generating space for
$Q_{4,1}(\eta')$ consists of order-four theta functions whose zeros
sum to $\frac12$ in $E_\tau$.  None of the proofs mentioned in the
previous paragraph address this fact in a direct way.
\end{remark}

\section{Point modules and line modules for \texorpdfstring{$A(\eta)$}{A(eta)}}
\label{sec:point-and-line-modules}

See \cite{SmithSurvey} for the definition of point and line modules.

The fact that the generators and relations for $A(\eta)$ are spaces of
theta functions allows one to give very concrete descriptions of point
modules and line modules for $A(\eta)$; see, for example,
\cite[\S2.4]{OdesskiiSurvey}.
Almost everything
in this section is already known
\cite{LS,OdesskiiSurvey,SmithSurvey,SmithStafford,SmithStaniszkis}.  We spell out the
details here for the benefit of the reader.

\begin{lemma}\label{lem:explicit-point-module}
Fix $x\in\CC$.  Let $M(x)$ be the vector space with basis
$\{e_k\mid k\geq0\}$, and, for $f\in V$, define
\[
 f\cdot e_k=f(x-k\eta)e_{k+1}.
\]
This action makes $M(x)$ a left $A(\eta)$-module generated by $e_0$.
\end{lemma}

\begin{proof}
The action of $V$ on $M(x)$ extends uniquely to a left
$TV$-module structure.  Let
\[
 r=\sum_i f_i\otimes g_i\in R_\eta.
\]
For every $k\geq0$,
\begin{align*}
 r\cdot e_k
 &=\sum_i f_i\cdot(g_i\cdot e_k)\\
 &=\sum_i
 f_i\bigl(x-(k+1)\eta\bigr)g_i(x-k\eta)e_{k+2}\\
 &=r\bigl(x-(k+1)\eta,x-k\eta\bigr)e_{k+2},
\end{align*}
which is $0$ because
$\bigl(x-(k+1)\eta,x-k\eta\bigr)\in\Gamma_\eta$.
Since $R_\eta$ annihilates $M(x)$, the $TV$-action factors through
$A(\eta)$.

It remains to prove that $e_0$ generates $M(x)$.  Suppose
inductively that $e_k\in A(\eta)e_0$.  There is an $f\in V$ such that
$f(x-k\eta)\ne0$; hence $f\cdot e_k=f(x-k\eta)e_{k+1}\ne0$; hence
$e_{k+1}\in A(\eta)e_0$; it follows that $M(x)=A(\eta)e_0$.
\end{proof}

Thus $M(x)$ is a point module for $A(\eta)$.  It is easy to show that
$M(x_1)\cong M(x_2)$ if and only if
$\overline{x_1}=\overline{x_2}$.  Thus, we often write
$M(\overline{x})$ for $M(x)$.  The modules $M(\overline{x})$ provide a
complete list of the isomorphism classes of the point modules for
$A(\eta)$ that are parametrized by $E_\tau$.

\begin{lemma}\label{lem:exchange-relation}
Every $r\in R_\eta$ satisfies the identity
\begin{equation}\label{eq:exchange-relation}
 \theta(z_1-z_2+2\eta)r(z_2+\eta,z_1)
 +\theta(z_1-z_2-2\eta)r(z_1+\eta,z_2)=0.
\end{equation}
\end{lemma}

\begin{proof}
Put $\nu=\frac12\eta$.  We use here the equality
$R_\eta=R_\eta^{\mathrm{RR}}$ from
Theorem~\ref{thm:Reta.equals.RR}.  Since $R_\eta^{\mathrm{RR}}$ is
the linear span of the functions $F_\eta(a,b,c_1,c_2)$ defined in
\eqref{eq.RR.F.definition}, it suffices to prove
\eqref{eq:exchange-relation} when
$r(w_1,w_2)=F_\eta(a,b,c_1,c_2;w_1,w_2)$ for arbitrary
$a,b,c_1,c_2\in\CC$.

We now use Weierstrass's three-term identity
\eqref{eq.Weierstrass.three-term} with
\begin{align*}
 w&=w_1+\nu-\tfrac12(a+b),&
 \alpha&=c_2+\tfrac12(a+b),\\
 \beta&=c_1+\tfrac12(a+b),&
 \gamma&=w_2-\nu-\tfrac12(a+b).
\end{align*}
After using the oddness of $\theta$, the resulting equality is the
Rains--Ruijsenaars factorization
\begin{align*}
&F_\eta(a,b,c_1,c_2;w_1,w_2)\\
&\quad=\theta(c_1+c_2+a+b)\theta(c_1-c_2)\\
&\qquad\cdot\theta(w_1+a-\nu)\theta(w_1+b-\nu)
       \theta(w_2+a+\nu)\theta(w_2+b+\nu)\\
&\qquad\cdot\theta(w_1+w_2-a-b)\theta(w_1-w_2+\eta).
\end{align*}
Set
\begin{align*}
H:={}&\theta(c_1+c_2+a+b)\theta(c_1-c_2)
       \theta(z_1+a+\nu)\theta(z_1+b+\nu)\\
&\cdot\theta(z_2+a+\nu)\theta(z_2+b+\nu)
       \theta(z_1+z_2+\eta-a-b).
\end{align*}
Substitution in the factorization gives
\begin{align*}
F_\eta(a,b,c_1,c_2;z_2+\eta,z_1)
 &=-H\theta(z_1-z_2-2\eta),\\
F_\eta(a,b,c_1,c_2;z_1+\eta,z_2)
 &= H\theta(z_1-z_2+2\eta).
\end{align*}
The two terms on the left side of \eqref{eq:exchange-relation} are
therefore negatives of one another.
\end{proof}

The next result shows that the modules $M(x,y)$ are line modules for
$A(\eta)$.  They are not all the line modules because of the condition
on $x-y$.  After the proposition we define line modules $M(x,y)$ for
all $x,y\in\CC$.

\begin{proposition}\label{prop:explicit-line-module}
Fix $x,y\in\CC$ and suppose that
\begin{equation}\label{eq:line-module-condition}
 x-y\notin\Lambda_\tau+2\ZZ\eta.
\end{equation}
Let $M(x,y)$ be the vector space with basis
$\{e_{ij}\mid i,j\geq0\}$.  For $f\in V$, define
{\small
\begin{equation}\label{eq:explicit-line-module-action}
 f\cdot e_{ij}:=
 \frac{f\bigl(x+(j-i)\eta\bigr)}
      {\theta\bigl(x-y+2(j-i)\eta\bigr)}e_{i+1,j}
 +
 \frac{f\bigl(y+(i-j)\eta\bigr)}
      {\theta\bigl(y-x+2(i-j)\eta\bigr)}e_{i,j+1}.
\end{equation}}
This action makes $M(x,y)$ a left $A(\eta)$-module generated by
$e_{00}$.
\end{proposition}

\begin{proof}
Condition \eqref{eq:line-module-condition} ensures that every
denominator in \eqref{eq:explicit-line-module-action} is nonzero.
Fix $i,j\geq0$ and put
\[
 p=x+(j-i)\eta,
 \qquad
 q=y+(i-j)\eta,
 \qquad
 \Delta=p-q.
\]
Let $r=\sum_\ell f_\ell\otimes g_\ell\in R_\eta$.  Applying
\eqref{eq:explicit-line-module-action} twice shows that the
coefficients of $e_{i+2,j}$ and $e_{i,j+2}$ in $r\cdot e_{ij}$ are
scalar multiples of $r(p-\eta,p)$ and $r(q-\eta,q)$,
both of which are zero because $(p-\eta,p)$ and $(q-\eta,q)$ belong to
$\Gamma_\eta$.

Since $\theta$ is odd, the coefficient of $e_{i+1,j+1}$ is
\[
 -\frac{1}{\theta(\Delta)}
 \left(
  \frac{r(q+\eta,p)}{\theta(\Delta-2\eta)}
  +\frac{r(p+\eta,q)}{\theta(\Delta+2\eta)}
 \right).
\]
The expression in parentheses is zero by
\eqref{eq:exchange-relation}.  Hence $r$ acts as zero on every
$e_{ij}$, and the resulting $TV$-action factors through $A(\eta)$.

It remains to prove that $e_{00}$ generates $M(x,y)$.  We argue by
induction on $i+j$.  Suppose that $e_{ij}\in A(\eta)e_{00}$.
Condition \eqref{eq:line-module-condition} implies that $p$ and $q$
represent distinct points of $E_\tau$.  Since $V$ separates points,
there are $f,g\in V$ such that
\[
 f(q)=0,\quad f(p)\ne0,
 \qquad
 g(p)=0,\quad g(q)\ne0.
\]
Formula \eqref{eq:explicit-line-module-action} then shows that
$f\cdot e_{ij}$ is a nonzero scalar multiple of $e_{i+1,j}$, while
$g\cdot e_{ij}$ is a nonzero scalar multiple of $e_{i,j+1}$.
Thus both successors belong to $A(\eta)e_{00}$.  Induction gives
$e_{ij}\in A(\eta)e_{00}$ for all $i,j\geq0$, and therefore
$M(x,y)=A(\eta)e_{00}$.
\end{proof}

If we make $M(x,y)$ a graded $A(\eta)$-module  by giving it the grading
$\deg(e_{ij})=i+j$, then it is a line module.  Its isomorphism class depends only on
$\overline{x}$ and $\overline{y}$, so we often denote it by
$M(\overline{x},\overline{y})$.  

Levasseur and Smith show that the line module associated to a secant line $L$ is isomorphic to $A(\eta)/A(\eta)L^\perp$ where $L^\perp
= \{v \in V \mid v(L)=0\}$. We therefore define, for  arbitrary $x,y\in\CC$, including those for which
\eqref{eq:line-module-condition} fails, 
\[
M(x,y) :=A(\eta)/A(\eta)L^\perp
\]
for all $x,y \in \CC$ where $L$ is the secant line such that $L \cap E_\tau=(\overline{x})+(\overline{y})$; equivalently, $L^\perp=H^0(E_\tau, {\mathcal L}_0(-\overline{x}-\overline{y})) \subset V$. 

For all $(x,y) \in \CC^2$, the isomorphism class of $M(x,y)$ depends only on $\overline{x}$ and $\overline{y}$ so we also denote it by $M(\overline{x},\overline{y})$.

\section{Central elements in \texorpdfstring{$A(\eta)_2$}{A(eta)2} and the
image of \texorpdfstring{$C_\eta$}{C-eta} in
\texorpdfstring{$A(\eta)_2$}{A(eta)2}}

\subsection{Centrality of the image of
\texorpdfstring{$C_\eta$}{C-eta}}

Suppose first that $\overline{\eta}\notin E_\tau[2]$.  Then the four
elements $X_{ab}$ in \eqref{eq.definition.Xab} form a basis for $V$.
In \cite[Thm.~2, p.~268]{Sklyanin1982}, Sklyanin said, without giving a proof, that the linearly
independent elements
\[
 K_0=\sum_{k=0}^3S_k^2,\qquad K_1=S_0^2+\sum_{k=1}^3(1-J_k)S_k^2
 =K_0-\sum_{k=1}^3\frac{\theta_{k+1}(2\eta)\theta_{k+1}(0)}{\theta_{k+1}(\eta)^2}S_k^2,
\]
belong to the center of $S(\eta)$.  
In \cite{CS}, a basis for the same central pencil is denoted
$\Omega_0,\Omega_1$ and, up to the normalizations used to identify the
generators, is written
\[
 \Omega_0=-x_0^2+x_1^2+x_2^2+x_3^2,
 \qquad
 \Omega_1=x_0^2+\beta\gamma x_1^2-\gamma x_2^2+\beta x_3^2,
\]
where their $\beta$ and $\gamma$ are our $\alpha_{01}$ and $\alpha_{10}$, respectively.

By definition,
$C_\eta\subset \operatorname{span}\{X_{00}^{\otimes2},\ldots,
X_{11}^{\otimes2}\}$; the next result gives an explicit expression for
an element of $C_\eta$ as a linear combination of the elements
$X_{ab}^{\otimes2}$.

\begin{proposition}\label{prop:diagonal-expansion}
Suppose that $\overline{\eta}\notin E_\tau[2]$.  For $\xi\in\CC$,
\begin{equation}\label{eq:diagonal-expansion}
 \theta(z_1-z_2\pm\eta)
 \theta(z_1+z_2\pm\xi)
 =
 \sum_{a,b\in\{0,1\}}
 \frac{
   \vartheta_{ab}(\xi+\eta)
   \vartheta_{ab}(\xi-\eta)}
  {2\vartheta_{ab}(\eta)^2}
 X_{ab}^{\otimes2},
\end{equation}
where $X_{ab}^{\otimes2}\in V^{\otimes2}$ is the function
$X_{ab}(z_1)X_{ab}(z_2)\in V(z_1)\otimes V(z_2)$.
\end{proposition}

\begin{proof}
The equality
$X_{ab}(z_1)X_{ab}(z_2)=(-1)^{a+b}\vartheta_{ab}(\eta)^2
\vartheta_{ab}(2z_1)\vartheta_{ab}(2z_2)$ follows from the definition
of $X_{ab}$.  Hence, if we set
\[
 (x_1,x_2,x_3,x_4)
 =
 (z_1-z_2+\eta,z_1-z_2-\eta,
  z_1+z_2+\xi,z_1+z_2-\xi),
\]
then \eqref{eq.Jacobi.five-term.four-theta} implies that
\begin{align*}
&2\theta(z_1-z_2+\eta)\theta(z_1-z_2-\eta)
 \theta(z_1+z_2+\xi)\theta(z_1+z_2-\xi)\\
={}&
 \vartheta_{11}(\xi+\eta)\vartheta_{11}(\xi-\eta)
 \vartheta_{11}(2z_1)\vartheta_{11}(2z_2)\\
&-
 \vartheta_{10}(\xi+\eta)\vartheta_{10}(\xi-\eta)
 \vartheta_{10}(2z_1)\vartheta_{10}(2z_2)\\
&+
 \vartheta_{00}(\xi+\eta)\vartheta_{00}(\xi-\eta)
 \vartheta_{00}(2z_1)\vartheta_{00}(2z_2)\\
&-
 \vartheta_{01}(\xi+\eta)\vartheta_{01}(\xi-\eta)
 \vartheta_{01}(2z_1)\vartheta_{01}(2z_2)\\
={}&
 \frac{\vartheta_{11}(\xi+\eta)\vartheta_{11}(\xi-\eta)}
      {\vartheta_{11}(\eta)^2}X_{11}^{\otimes2}
 +\frac{\vartheta_{10}(\xi+\eta)\vartheta_{10}(\xi-\eta)}
      {\vartheta_{10}(\eta)^2}X_{10}^{\otimes2}\\
&+
 \frac{\vartheta_{00}(\xi+\eta)\vartheta_{00}(\xi-\eta)}
      {\vartheta_{00}(\eta)^2}X_{00}^{\otimes2}
 +\frac{\vartheta_{01}(\xi+\eta)\vartheta_{01}(\xi-\eta)}
      {\vartheta_{01}(\eta)^2}X_{01}^{\otimes2}.
\end{align*}
Dividing by $2$ gives \eqref{eq:diagonal-expansion}.
\end{proof}

\begin{lemma}\label{lem.Omega.tilde.basis.Ceta}
Suppose $\overline\eta\notin E_\tau[2]$, and put
$\beta=\alpha_{01}$ and $\gamma=\alpha_{10}$.  If
$\overline\eta\notin E_\tau[4]$, retain the indexing of the generators
$X_{11},X_{00},X_{01},X_{10}$ fixed in \S2.1.1.  If
$\overline\eta$ has exact order four, make the cyclic relabelling in
Proposition~\ref{prop:order-four-relations}, so that
$\beta=1$ and $\gamma=-1$.  Then the tensors
\begin{align*}
 \widetilde\Omega_0
 &:=-X_{11}^{\otimes2}+X_{00}^{\otimes2}
       +X_{01}^{\otimes2}+X_{10}^{\otimes2},\\
 \widetilde\Omega_1
 &:=X_{11}^{\otimes2}+\beta\gamma X_{00}^{\otimes2}
       -\gamma X_{01}^{\otimes2}+\beta X_{10}^{\otimes2}
\end{align*}
form a basis for $C_\eta=K_{\eta,00}$.
\end{lemma}

\begin{proof}
Proposition~\ref{prop:diagonal-expansion} with $\xi=0$ gives
$\theta(z_1-z_2\pm\eta)\theta(z_1+z_2)^2
=\tfrac12\widetilde\Omega_0$.
With $\xi=\tfrac12(1+\tau)$, the standard half-period transformation
formulas \cite[Table~0, p.~19]{MumfordTataI} give, on putting
\[
 c_{ab}:=
 \frac{\vartheta_{ab}(\xi+\eta)\vartheta_{ab}(\xi-\eta)}
 {2\vartheta_{ab}(\eta)^2},
 \qquad \xi=\tfrac12(1+\tau),
\]
the following coefficient vector in \eqref{eq:diagonal-expansion},
with respect to the ordered basis
$(X_{11}^{\otimes2},X_{00}^{\otimes2},X_{01}^{\otimes2},
X_{10}^{\otimes2})$:
\[
 (c_{11},c_{00},c_{01},c_{10})
 =c_{11}(1,\beta\gamma,-\gamma,\beta).
\]
Thus the corresponding left-hand
side is a nonzero scalar multiple of $\widetilde\Omega_1$.
Both left-hand sides belong to $C_\eta$ by
Proposition~\ref{thm.dim.Ceta}(3), so
$\widetilde\Omega_0,\widetilde\Omega_1\in C_\eta=K_{\eta,00}$.

Their coefficient vectors are linearly independent.  Indeed,
proportionality would give $\beta=-1$ and $\gamma=1$.  This is
impossible off $E_\tau[4]$ by
Corollary~\ref{cor.alpha.parameters.and.four.torsion}, and is excluded
by the chosen relabelling at exact order four.  Since $C_\eta$ has
dimension two by Proposition~\ref{thm.dim.Ceta}, the result follows.
\end{proof}

\begin{remark}
We now assume that $\overline\eta\notin E_\tau[4]$, put
$\alpha=\alpha_{00}$, $\beta=\alpha_{01}$, and
$\gamma=\alpha_{10}$, and identify
$(x_0,x_1,x_2,x_3)=(x_{11},x_{00},x_{01},x_{10})$.  Define
\begin{equation}\label{eq.Stafford.central.elements}
 \Omega_1^{\rm St}:=-x_0^2+x_1^2+x_2^2+x_3^2,
 \qquad
 \Omega_2^{\rm St}:=x_1^2+\frac{1+\alpha}{1-\beta}x_2^2
                 +\frac{1-\alpha}{1+\gamma}x_3^2.
\end{equation}
Stafford denotes these elements by $\Omega_1$ and $\Omega_2$ in
\cite[(0.4.1)]{StaffordRegularity}; they are the images of
\begin{equation}\label{eq.Stafford.central.lifts}
 \widetilde\Omega_1^{\rm St}:=\widetilde\Omega_0,
 \qquad
 \widetilde\Omega_2^{\rm St}
 :=\frac{\widetilde\Omega_0+\widetilde\Omega_1}{1+\beta\gamma},
\end{equation}
both of which belong to $C_\eta$; the formula for
$\widetilde\Omega_2^{\rm St}$ has the stated image because
$(1+\alpha)/(1-\beta)=(1-\gamma)/(1+\beta\gamma)$ and
$(1-\alpha)/(1+\gamma)=(1+\beta)/(1+\beta\gamma)$.
We will use \eqref{eq.Stafford.central.elements} and
\eqref{eq.Stafford.central.lifts} in \cite{StaffordRelations}.
\end{remark}

\begin{theorem}\label{thm:Ceta-central-all}
For every $\eta\in\CC$, the image of $C_\eta$ in $A(\eta)_2$ consists
of central elements.
\end{theorem}

\begin{proof}
We divide the proof into four cases.

\smallskip
\noindent
\emph{(1) $\overline\eta=0$.}
Here $A(\eta)=SV$, so every element is central.

\smallskip
\noindent
\emph{(2) $\overline\eta$ has order $2$.}
The elements $Y_{ab}$ provide the displayed bases of $V^+$ and
$V^-$ in Theorem~\ref{prop:addition-formulas-relations}(2).  Their
four squares are central: each commutes with the two generators in its
own eigenspace and with the two generators in the other eigenspace.
Indeed, if generators $x$ and $y$ anticommute, then
$x^2y=x(xy)=-x(yx)=yx^2$; thus anticommutation with a generator implies
commutation of its square with that generator.
Since $D=\operatorname{span}\{Y_{ab}\otimes Y_{ab}\mid
a,b\in\{0,1\}\}$ and $C_\eta\subset D$,
every element of $C_\eta$ has central image in $A(\eta)$.

\smallskip
\noindent
\emph{(3) $\overline\eta$ has order $4$.}
Use the cyclic relabelling in Proposition~\ref{prop:order-four-relations}.
Let $P=x_0^2-x_1^2$ and $Q=x_2^2+x_3^2$, and let
$\widetilde P,\widetilde Q\in V^{\otimes2}$ be their diagonal tensor
representatives.  Here $\beta=1$ and $\gamma=-1$, so
\[
 \widetilde\Omega_0=-\widetilde P+\widetilde Q,
 \qquad
 \widetilde\Omega_1=\widetilde P+\widetilde Q.
\]
Lemma~\ref{lem.Omega.tilde.basis.Ceta} therefore gives
$C_\eta=K_{\eta,00}=\CC\widetilde\Omega_0\oplus
\CC\widetilde\Omega_1=\CC\widetilde P\oplus\CC\widetilde Q$.
The six relations in that proposition give
\begin{align*}
 [x_0,x_1^2]&=[x_1,x_0^2]=[x_2,x_3^2]=[x_3,x_2^2]=0,\\
 [x_0,Q]&=[x_1,Q]=[x_2,P]=[x_3,P]=0.
\end{align*}
For example,
\begin{align*}
[x_0,x_1^2]
 &=\alpha\bigl(\{x_2,x_3\}x_1+x_1\{x_2,x_3\}\bigr)=0,\\
[x_0,Q]
 &=(x_3x_1x_2+x_2x_1x_3)
   -(x_2x_1x_3+x_3x_1x_2)=0;
\end{align*}
the other identities follow in the same way.  Thus $P$ and $Q$ are
central, and the displayed description of $C_\eta$ proves the assertion
in this case.

\smallskip
\noindent
\emph{(4) $\overline\eta\notin E_\tau[4]$.}
Lemma~\ref{lem.Omega.tilde.basis.Ceta} gives
$C_\eta=K_{\eta,00}=\CC\widetilde\Omega_0\oplus
\CC\widetilde\Omega_1$.
After identifying the generators, Proposition~6.1 of \cite{CS} says
that the images $\Omega_0,\Omega_1$ of these two tensors are central.
Hence the image of $C_\eta$ is central.
\end{proof}

The order-four case supplies precisely the step missing from the proof
in \cite[Prop.~6.1]{CS} at nonzero four-torsion: a coefficient by which
that proof divides can vanish there.  Proposition
\ref{prop:order-four-relations} gives the direct argument above.

\subsection{Explicit representatives and bases}

Define
\[
F_{ab}(z,\eta)=\vth_{ab}(\eta)^2
\vth_{ab}(2z)\vth_{ab}(2z+2\eta).
\]

Let $h_0,h_1,h_2,h_3$ denote the right-hand sides of
Rains--Ruijsenaars~\cite[(B.23)--(B.26)]{RR}, respectively. In our
notation,
\[
\begin{aligned}
h_0(z)&=\theta_1(\eta)\theta_1(2z),&
h_1(z)&=\theta_2(\eta)\theta_2(2z),\\
h_2(z)&=i\theta_3(\eta)\theta_3(2z),&
h_3(z)&=\theta_4(\eta)\theta_4(2z).
\end{aligned}
\]
Thus,
\[
\begin{aligned}
F_{11}(z,\eta)&=\phantom{-}h_0(z)h_0(z+\eta),&
F_{10}(z,\eta)&=\phantom{-}h_1(z)h_1(z+\eta),\\
F_{00}(z,\eta)&=-h_2(z)h_2(z+\eta),&
F_{01}(z,\eta)&=\phantom{-}h_3(z)h_3(z+\eta).
\end{aligned}
\]
Let $u=2z+\eta$ and $v=\eta$, and put $X=\theta_1(u)^2$ and
$Y=\theta_2(u)^2$.
The addition identities \eqref{eq:add1}--\eqref{eq:add4} can be written
as the matrix equation
\[
\theta_2(0)^2\begin{pmatrix}F_{11}&F_{10}&F_{00}&F_{01}\end{pmatrix}
=\begin{pmatrix}X&Y\end{pmatrix}M,
\]
where
\[
M=\begin{pmatrix}
AB&-AB&CD&CD\\
-A^2&B^2&C^2&D^2
\end{pmatrix},
\]
and
\[
A=\vth_{11}(\eta)^2,\qquad B=\vth_{10}(\eta)^2,
\qquad C=\vth_{00}(\eta)^2,
\qquad D=\vth_{01}(\eta)^2.
\]
Jacobi's quartic identity in
\eqref{eq.Jacobi.quartic.identity.intro} gives
$A^2+C^2=B^2+D^2$; put
\[
\Delta=B^2-A^2=C^2-D^2.
\]
Thus the nullspace of $M$ is the full space of linear relations among the
four functions $F_{ab}$.

\begin{theorem}\label{thm:explicit-C}
Assume $\overline{\eta}\notin E_\tau[2]$.  The following five elements
belong to $C_\eta$:
\begin{align*}
&BC(AC+BD)X_{11}^{\otimes2}
-AC(BC+AD)X_{10}^{\otimes2}
-AB\Delta X_{00}^{\otimes2},\\
&BD(AD+BC)X_{11}^{\otimes2}
-AD(BD+AC)X_{10}^{\otimes2}
+AB\Delta X_{01}^{\otimes2},\\
&CD\Delta X_{11}^{\otimes2}
+AD(BD+AC)X_{00}^{\otimes2}
+AC(BC+AD)X_{01}^{\otimes2},\\
&CD\Delta X_{10}^{\otimes2}
+BD(AD+BC)X_{00}^{\otimes2}
+BC(AC+BD)X_{01}^{\otimes2},\\
&X_{11}^{\otimes2}-X_{10}^{\otimes2}
-X_{00}^{\otimes2}-X_{01}^{\otimes2}.
\end{align*}
Equivalently, evaluating each $X_{ab}^{\otimes2}$ at $(z,z+\eta)$ makes
each of the five displayed expressions equal to zero for
all $z$.
\end{theorem}

\begin{proof}
For $1\leq j\leq4$, let $M_j$ denote the $j$th column of $M$.
These columns correspond, in order, to
$F_{11},F_{10},F_{00},F_{01}$. For any
three $2\times1$ columns $p,q,r$,
\[
\det(q,r)p-\det(p,r)q+\det(p,q)r=0.
\]
Applying this to the four triples of columns of $M$, and using
$X_{ab}(z)X_{ab}(z+\eta)=\gamma_{ab}^2F_{ab}(z,\eta)$, gives
the first four displayed relations. Subtracting the relation obtained from
$(M_1,M_2,M_4)$ from that obtained from $(M_1,M_2,M_3)$, and using
$\Delta=C^2-D^2$, gives the fifth displayed relation.
By the definition of $C_\eta$, these relations belong to $C_\eta$.
\end{proof}

By Theorem~\ref{thm:Ceta-central-all}, the images in
$A(\eta)$ of the five tensors displayed in
Theorem~\ref{thm:explicit-C} are central.

\begin{corollary}\label{cor:bases-C}
Whenever two of the five tensors in Theorem~\ref{thm:explicit-C} are
linearly independent, the set consisting of them is a basis for
$C_\eta$. If
$\overline{\eta}\notin E_\tau[2]$,
$M$ has rank $2$, so its nullspace has dimension $2$ and any two
independent relations among the five displayed elements give such a
basis.
\end{corollary}

\begin{proof}
Proposition~\ref{thm.dim.Ceta} gives $\dim C_\eta=2$ for every $\eta$.
\end{proof}

\begin{remark}
After the tensors $X_{ab}$ are replaced by their classes $x_{ab}$ in
$A(\eta)$ and $A(\eta)$ is identified with $S(\eta)$, the five elements in
Theorem~\ref{thm:explicit-C} are, up to normalization and relabelling,
the five quadratic forms displayed by Chirvasitu--Smith
\cite[Prop.~2.8]{CSExotic}.  Proposition~2.8 of that paper does not
assign names to those forms.  Rains--Ruijsenaars'
$K_0,K_2$ \cite[(B.36)--(B.38)]{RR} give another basis for the same
space; after translating notation, they are Sklyanin's $K_0,K_1$.
Smith--Stafford use the basis $\Omega,\Omega'$ and Smith--Staniszkis
use the same central pencil.  Thus all these named elements span the
same two-dimensional subspace of $A(\eta)_2$; their formulas differ
only because the cited papers use different bases and normalizations.
\end{remark}

\subsection{The geometric meaning of \texorpdfstring{$C_0$}{C-0}}

Let $SV$ be the symmetric algebra on $V$, so
$\PP V^*=\operatorname{Proj}(SV)$, and let $I$ be the homogeneous ideal
in $SV$ that vanishes on $E_\tau$.  The quadrics that contain
$E_\tau$ are the zero loci of the nonzero elements in $I_2$, the
quadratic component of $I$.

\begin{proposition}
The inverse image of $I_2$ under the map $D\to S^2V$ is
\[
C_0=\theta(z_1-z_2)^2W(z_1+z_2).
\]
\end{proposition}

\begin{proof}
The image in $S^2V$ of an element of $D$ belongs to $I_2$ precisely when
the corresponding quadric vanishes on $E_\tau$. In the function-theoretic
realization of $D$, this says precisely that the element vanishes on the
diagonal $z_1=z_2$, which is the defining condition for membership in
$C_0$. The displayed description follows from Proposition~\ref{thm.dim.Ceta}
with
$\eta=0$.
\end{proof}

The geometric description in
Proposition~\ref{prop.geometric.vertices}, specialized to $c=0$, labels
them $Q_0(w)$, with $Q_0(w)=Q_0(-w)$, and shows that the lines that lie
on them are precisely the secant lines to $E_\tau$.

If $\overline\eta\ne0$, Proposition~\ref{prop:Cnu-Ceta}, with
\(\nu=0\), gives
\[
 C_0\cap C_\eta=\{0\},
 \qquad D=C_0\oplus C_\eta.
\]

\begin{remark}
When $\overline\eta=0$, the quotient
$TV/(R_\eta\oplus C_\eta)$ is the homogeneous coordinate ring of
$E_\tau$ for the embedding defined by
$\mathcal O_{E_\tau}(4[0])$.  At a nonzero two-torsion point the direct
sum above still holds, although the $\eta$-dependent normalization by
the functions $X_{ab}$ can degenerate.
\end{remark}

\subsection{An explicit formula for the central element
\texorpdfstring{$\Omega(\zeta)$}{Omega(zeta)}}

Throughout this subsection, we assume that
$\overline\eta\notin E_\tau[4]$.  The algebras at the four-torsion
parameters have rather different behavior and will not be considered
here.

Fix a point $\zeta \in \CC$. By \cite[Cor.~6.6]{LS},\footnote{The translation parameter in
\cite{LS} is denoted by $\tau$, but we will continue to use $\eta$.}
there is a non-zero central element $\Omega(\overline{\zeta})$ in
$S(\eta)_2$ that annihilates every line module $M(p,q)$ for which
$p+q=\overline{\zeta}$.  See also \cite[\S9, especially Thm.~9.1]{SmithSurvey}.
An $\Omega(\overline{\zeta})$ with this property is only
determined up to a nonzero scalar multiple.  Levasseur and Smith also
show that
\[
 \CC\Omega(\overline{\zeta})=\CC\Omega(-\overline{\zeta}-2\overline{\eta}).
\]
Using the isomorphism $S(\eta)\cong A(\eta)$, we regard these central
elements and the equality of their one-dimensional spans as statements about
$A(\eta)$.  In this section we
determine explicit elements $\Omega(\overline{\zeta})\in A(\eta)_2$ with these
properties.

We now define
\begin{align}
 \Omega_\zeta(z_1,z_2)
 &:=\theta(z_1-z_2\pm\eta)\,
 \theta\bigl(z_1+z_2\pm(\zeta+\eta)\bigr)
 \quad\in C_\eta,
 \label{eq.Omega-zeta-function}\\
 \Omega(\zeta)&:=\Omega_\zeta+R_\eta
 \quad\in A(\eta)_2.
 \label{eq.explicit.Omega.z}
\end{align}
The tensor $\Omega_\zeta$ belongs to $C_\eta$, so its image in
$A(\eta)_2$ is nonzero and central.  Clearly,
\[
 \Omega(\zeta)=\Omega(-\zeta-2\eta),
\]

\begin{theorem}\label{thm.explicit.Omega.z}
Suppose that $\overline\eta\notin E_\tau[4]$.
If $p,q\in E_\tau$, $\zeta\in\CC$,
and $p+q=\overline{\zeta}$, then the element $\Omega(\zeta)$ defined in
\eqref{eq.explicit.Omega.z} annihilates the line module $M(p,q)$.
Thus \eqref{eq.explicit.Omega.z} gives a specific choice for the
element whose existence is proved by Levasseur--Smith.
\end{theorem}

\begin{proof}
Choose preimages \(x,y\in\CC\) of \(p,q\), respectively, such that
\(x+y=\zeta\).  Since \(\Omega(\zeta)\) is central and \(M(x,y)\) is
generated by its degree-zero element \(e_{00}\), it suffices to show
that \(\Omega(\zeta)\cdot e_{00}=0\).  Suppose first that
\begin{equation}\label{eq:Omega-proof-condition}
 x-y\notin\Lambda_\tau+2\ZZ\eta.
\end{equation}
Proposition~\ref{prop:explicit-line-module} gives the basis and action
in \eqref{eq:explicit-line-module-action} for \(M(x,y)=M(p,q)\).
Put \(\Delta=x-y\).  Applying \eqref{eq:explicit-line-module-action}
twice gives,
for \(h\in V^{\otimes2}\),
{\small
\[
 h\cdot e_{00}
 =\frac{h(x-\eta,x)}{\theta(\Delta)\theta(\Delta-2\eta)}e_{20}
 -\frac{1}{\theta(\Delta)}\left(
  \frac{h(y+\eta,x)}{\theta(\Delta-2\eta)}
  +\frac{h(x+\eta,y)}{\theta(\Delta+2\eta)}
 \right)e_{11}
 +\frac{h(y-\eta,y)}{\theta(\Delta)\theta(\Delta+2\eta)}e_{02}.
\]}
Since \(\theta(z_1-z_2\pm\eta)\) is a factor of
\(\Omega_\zeta(z_1,z_2)\),
\(\Omega_\zeta(x-\eta,x)
=\Omega_\zeta(y-\eta,y)=0\); thus, when
\(h=\Omega_\zeta\), the first and third terms in
\(h\cdot e_{00}\) vanish.  Similarly, because
\(\Omega_\zeta(z_1,z_2)\) is a multiple of
\(\theta\bigl(z_1+z_2\pm(\zeta+\eta)\bigr)\),
\(\Omega_\zeta(y+\eta,x)
=\Omega_\zeta(x+\eta,y)=0\), which implies that the
middle term vanishes when \(h=\Omega_\zeta\).
Hence \(\Omega(\zeta)\cdot e_{00}=0\), and centrality shows that
\(\Omega(\zeta)\) annihilates \(M(p,q)\) whenever
\eqref{eq:Omega-proof-condition} holds.

We will now remove that condition.
Fix \(\zeta\), let \(x\) vary, put \(y=\zeta-x\), and write
\[
 W_x:=W_{x,\zeta-x}\subset V,
 \qquad
 U_x:=R_\eta+V\otimes W_x\subset V^{\otimes2}.
\]
The assignment $\overline x\mapsto W_x$ is a morphism
$E_\tau\to\operatorname{Gr}(2,V)$: it assigns to $\overline x$ the
two-dimensional space of linear forms vanishing on the secant or tangent
line whose intersection divisor with $E_\tau$ is
$(\overline x)+(\overline\zeta-\overline x)$.  Every corresponding quotient
$A(\eta)/A(\eta)W_x$ is a line module, so its degree-two component has
dimension three.  Consequently, $\dim U_x=\dim V^{\otimes2}-3=13$ for
every $x$, and $\overline x\mapsto U_x$ is a morphism
$E_\tau\to\operatorname{Gr}(13,V^{\otimes2})$.

Let $Z_\zeta$ be the inverse image under this map of the Schubert
subvariety consisting of the $13$-planes that contain the fixed tensor
$\Omega_\zeta$.  Thus $Z_\zeta$ is closed, and
\[
 \overline x\in Z_\zeta
 \quad\Longleftrightarrow\quad
 \Omega_\zeta\in R_\eta+V\otimes W_x
 \quad\Longleftrightarrow\quad
 \Omega(\zeta)\cdot e_{00}=0
\]
in $M(x,\zeta-x)$.  The calculation above shows that $Z_\zeta$ contains
every $\overline x$ such that
\[
 2x-\zeta\notin\Lambda_\tau+2\ZZ\eta.
\]
The images of these \(x\) in \(E_\tau\), i.e., the points
\(\overline{x}\), form a dense subset: the excluded subset is at most
countable.  Hence the closed set $Z_\zeta$ is all of $E_\tau$.
Therefore \(\Omega(\zeta)\) annihilates the degree-zero generator of every
\(M(x,\zeta-x)\), including those for which
\eqref{eq:Omega-proof-condition} fails.  By centrality it annihilates
each such line module.
\end{proof}

With $K_0$ and $K_2$ defined as in
\cite[(B.37)--(B.38)]{RR}, substitution in
\eqref{eq.explicit.Omega.z} gives
\begin{equation}\label{eq.Omega.specializations}
 \Omega(-\eta)=\tfrac12K_0,
 \qquad
 \Omega(0)=\tfrac12K_2.
\end{equation}

\subsubsection{The \texorpdfstring{$\Omega(\zeta)$}{Omega(zeta)}
whose zero-loci are noncommutative quadrics in the sense of
Smith--Van den Bergh \cite{SmithVdBQuadrics}}

For a connected graded noetherian ring \(R\), write
\({\rm Proj}_{\sf nc}(R)\) for the noncommutative scheme whose
category of quasi-coherent sheaves is the quotient category associated
to \(R\), as in \cite[\S\S1.13, 1.14]{CSExotic}.  We consider
$\PP_{\sf Skly}^3:={\rm Proj}_{\sf nc}\bigl(A(\eta)\bigr)$ as a
noncommutative analogue of \(\PP^3\).  This space has a closed
subvariety isomorphic to \(E_\tau\), namely
\({\rm Proj}_{\sf nc}\bigl(A(\eta)/(C_\eta)\bigr)\).

Similarly, we define
\[
 Q_\zeta:=
 {\rm Proj}_{\sf nc}\bigl(A(\eta)/(\Omega(\zeta))\bigr)
\]
and regard these as a pencil of noncommutative quadrics in
\(\PP_{\sf Skly}^3\) that contain this copy of \(E_\tau\).  The purpose
of \cite{SmithVdBQuadrics} is to show that the noncommutative quadrics
\(Q_\zeta\) behave like the pencil of commutative quadrics that
contain the copy of \(E_\tau\) in \(\PP^3\).
One of their main results is that exactly four of the \(Q_\zeta\) are
singular.

\begin{proposition}\label{prop.Omega.fixed.points}
Assume \(\overline\eta\notin E_\tau[4]\).  For
\(\zeta\in\CC\), the following statements are equivalent.
\begin{enumerate}
\item The noncommutative quadric \(Q_\zeta\) is singular in the sense
of Smith--Van den Bergh.
\item \(\overline\zeta+\overline\eta\in E_\tau[2]\).
\item If \(\Omega_\zeta(z_1,z_2)\) is the function in
\eqref{eq.Omega-zeta-function}, then the intersection
\(\mathcal Z(\Omega_\zeta)\cap(E_\tau\times E_\tau)\) contains the addition fiber
$D_{\zeta+\eta}:=\{(p,q)\in E_\tau\times E_\tau\mid
   p+q=\overline\zeta+\overline\eta\}$
with multiplicity two.
\item The point \(\CC\Omega(\zeta)\) is a branch point of the
degree-two map
\[
 E_\tau\longrightarrow\PP(C_\eta),\qquad
 \overline\zeta\longmapsto\CC\Omega(\zeta).
\]
\item Up to a nonzero scalar multiple, \(\Omega(\zeta)\) is one of the
following four elements:
putting
\[
 \alpha=\alpha_{00}=\frac{AC}{BD},\qquad
 \beta=\alpha_{01}=-\frac{AD}{BC},\qquad
 \gamma=\alpha_{10}=-\frac{AB}{CD},
\]
these elements are
\begin{align*}
 \Theta_0&:=-x_0^2+x_1^2+x_2^2+x_3^2,\\
 \Theta_1&:=x_0^2+\beta\gamma\,x_1^2-\gamma\,x_2^2+\beta\,x_3^2,\\
 \Theta_2&:=x_0^2+\gamma\,x_1^2+\alpha\gamma\,x_2^2-\alpha\,x_3^2,\\
 \Theta_3&:=x_0^2-\beta\,x_1^2+\alpha\,x_2^2+\alpha\beta\,x_3^2.
\end{align*}
These are the four central elements in \cite[Prop.~2.8]{CSExotic}.
\end{enumerate}
\end{proposition}

\begin{proof}
The equivalence of (1) and (2) is given by
\cite[Thm.~10.2]{SmithVdBQuadrics}.

Put \(a=\zeta+\eta\).  By \eqref{eq.Omega-zeta-function}, the second
theta factor in \(\Omega_\zeta(z_1,z_2)\) is
\(\theta(z_1+z_2+a)\theta(z_1+z_2-a)\).
Its two zero divisors on \(E_\tau\times E_\tau\) coincide if and only if
\(\overline a=-\overline a\), i.e., if and only if
\(\overline a\in E_\tau[2]\); i.e., if and only if its associated
divisor contains \(D_a\) with multiplicity two.  Conversely, \(D_a\) can
occur with multiplicity two only when the two addition fibers
defined by the second theta factor coincide.  This proves the
equivalence of (2) and (3).

The map in (4) is the quotient map for the involution
\(\overline\zeta\mapsto-\overline\zeta-2\overline\eta\).
Its ramification points are precisely the fixed points of this
involution, and the fixed-point equation is
\(2(\overline\zeta+\overline\eta)=0\).  Thus (2) and (4) are
equivalent.

Finally, write
\(E_\tau[2]=\{0,\xi_1,\xi_2,\xi_3\}\), with \(\xi_0=0\), and put
\(\zeta_j=-\overline\eta+\xi_j\).  These are precisely the four parameters
satisfying (2).  For \(j=0\),
Proposition~\ref{prop:diagonal-expansion}, with \(\xi=0\), gives one
half of the first tensor in (5).  For completeness, apply the standard
half-period transformation formulas to the coefficients in
\eqref{eq:diagonal-expansion}.  With respect to the ordered basis
\((X_{11}^{\otimes2},X_{00}^{\otimes2},
X_{01}^{\otimes2},X_{10}^{\otimes2})\) of $D$, whose images in
$A(\eta)_2$ are $(x_0^2,x_1^2,x_2^2,x_3^2)$, the resulting coefficient
vectors, with a common nonzero factor in each row suppressed, are
\[
\begin{array}{c|c}
 \xi&\text{coefficient vector of }\Omega(-\eta+\xi)\\ \hline
 0&(-1,1,1,1)\\
 \tfrac12&(1,-\beta,\alpha,\alpha\beta)\\
 \tfrac12\tau&(1,\gamma,\alpha\gamma,-\alpha)\\
 \tfrac12(1+\tau)&(1,\beta\gamma,-\gamma,\beta).
\end{array}
\]
For example, translation by (1/2) interchanges
\(\vartheta_{11}\) with \(\vartheta_{10}\) and
\(\vartheta_{00}\) with \(\vartheta_{01}\), with the signs dictated
by parity.  Substitution of
\(\alpha=AC/(BD)\), \(\beta=-AD/(BC)\), and
\(\gamma=-AB/(CD)\) gives the second row; the transformations by
\(\tau/2\) and \((1+\tau)/2\) give the last two rows in the same way.
Thus the four rows are exactly the four tensors displayed in (5), with
all signs and coefficients accounted for.  Conversely, the four projective classes of
central elements in (5) are the four branch points of the
degree-two map in (4), so (5) is equivalent to (4).  They agree,
up to scalar, with the four equations in
\cite[Prop.~2.8]{CSExotic}.
\end{proof}

\section{Automorphisms of \texorpdfstring{$A(\eta)$}{A(eta)}}
\label{sec:automorphisms}

First, some comments about notation.  In the paper of Smith--Staniszkis
\cite[Section~2]{SmithStaniszkis},
the subgroup $E_\tau[4]$ is denoted by $E_4$.  Smith--Staniszkis write
$\eta$ for the elliptic modulus and $\tau$ for the translation parameter.
In this paper, $\tau$ denotes the elliptic modulus and $\eta$ denotes
the translation parameter.

If $\xi\in E_\tau[4]$, i.e., if $4\xi=0$, then
$t_\xi^*\mathcal L\cong\mathcal L$ for every degree-four line bundle
$\mathcal L$ on $E_\tau$, where $t_\xi:E_\tau\to E_\tau$ is translation
by $\xi$.  In particular, this
holds for $\mathcal L=\mathcal O_{E_\tau}(4[0])$, so $t_\xi$
induces a projective transformation of $\PP V^*$ whose restriction to
$E_\tau$ is $p\mapsto p+\xi$.  Choose
$\Phi(\xi)\in\operatorname{GL}(V)$ so that the projective
transformation induced by $(\Phi(\xi)^{-1})^*$ on $\PP V^*$ equals
$t_\xi$.  Thus $\Phi(\xi)$ is the contragredient linear lift
appropriate for the action on the generating space $V$; it is determined
up to multiplication by a nonzero scalar.

\begin{proposition}
\label{prop.SS.automorphisms}
For every $\eta\in\CC$ and every $\xi\in E_\tau[4]$, the linear map
$\Phi(\xi)$ extends to
a graded automorphism of $A(\eta)$.  The map
\[
 E_\tau[4]\longrightarrow
 \operatorname{Aut}\bigl(A(\eta)\bigr)/\CC^\times,
 \qquad
 \xi\longmapsto[\Phi(\xi)],
\]
is a group homomorphism, and its induced action on $E_\tau$ is
translation by $\xi$ through the inverse-dual projective action
of $(\Phi(\xi)^{-1})^*$.
\end{proposition}

\begin{proof}
Put $s=\Phi(\xi)$ and let $\widehat s=(s^{-1})^*$ denote its action
on $\PP V^*$.  For $r\in V^{\otimes2}$ and $p,q\in V^*$ one has
\[
 ((s\otimes s)r)(p,q)=r(s^*p,s^*q)
 =r(\widehat s^{-1}p,\widehat s^{-1}q).
\]
By the choice of $s$, the projective transformation $\widehat s$
restricts to $p\mapsto p+\xi$ on $E_\tau$ and permutes the four vertices
$e_0,e_1,e_2,e_3$.  It therefore preserves both components of the
vanishing locus used to define $R_\eta$, namely $\Gamma_\eta$ and the four
vertex pairs:
\[
 (p,p+\overline\eta)\longmapsto
 (p+\xi,p+\overline\eta+\xi),
 \qquad
 (e_i,e_i)\longmapsto(e_j,e_j).
\]
Hence, $(\Phi(\xi)\otimes\Phi(\xi))(R_\eta)=R_\eta$, so
$\Phi(\xi)$ extends to an automorphism of
$TV/(R_\eta)=A(\eta)$.

For $\xi,\xi'\in E_\tau[4]$, the inverse-dual projective transformations
induced by $\Phi(\xi)\Phi(\xi')$ and $\Phi(\xi+\xi')$ both restrict on
$E_\tau$ to translation by $\xi+\xi'$.  Since $E_\tau$ spans
$\PP V^*$, a projective transformation
of $\PP V^*$ is determined by its restriction to $E_\tau$: a linear
lift of a transformation fixing $E_\tau$ pointwise has every point of
$E_\tau$ as an eigenline, so irreducibility forces $E_\tau$ to lie in
one eigenspace, and nondegeneracy forces that eigenspace to be all of
$V^*$.  Hence $[\Phi(\xi)\Phi(\xi')]=[\Phi(\xi+\xi')]$ in
$\operatorname{PGL}(V)$.  This proves the group-homomorphism
assertion.
\end{proof}

\begin{remark}
Notice that passage to the quotient by $\CC^\times$ is essential: the chosen
linear lifts form a Heisenberg central extension and need not multiply
exactly.  Smith--Staniszkis prove that these projective $E_\tau[4]$-actions
lift to automorphisms of $S(\eta)$ in
\cite[Thm.~2.2(a),(b)]{SmithStaniszkis}; as they
note at the bottom of \cite[p.~63]{SmithStaniszkis}, their verification
that the displayed maps preserve the relations uses the criterion in
\cite[Prop.~6, eq.~(25)]{Sklyanin1983}, which characterizes certain
automorphisms by identities among their scalar coefficients.  A related
use of this criterion appears in \cite[\S6]{CS}.  Here we use instead
the following principle: if $s\in\operatorname{GL}(V)$ is such that the
projective transformation induced by $(s^{-1})^*$ preserves the translation
graph and permutes the four vertices, then
$(s\otimes s)(R_\eta)=R_\eta$, so $s$ extends to a graded automorphism of
$A(\eta)$.  This gives a geometric proof that avoids the coefficient
calculation.  Finally, Theorem~\ref{prop:addition-formulas-relations}(1) identifies
$A(\eta)$ with $S(\eta)$, so the present proposition
recovers the Smith--Staniszkis result under that identification.
\end{remark}

We now record explicit choices of these lifts under the assumption that
$\overline{\eta}\notin E_\tau[2]$.  Recall the elements $X_{ab}$ defined
in \eqref{eq.definition.Xab}, and use the ordered basis
$(X_{11},X_{00},X_{01},X_{10})$ of $V$.  Smith--Staniszkis use
Weber's theta functions.  Weber's
$\vth_{11}$ is the negative of Mumford's, while the other three theta
functions agree.  The function $X_{11}$ itself is unchanged, because
both factors $\vth_{11}(\eta)$ and $\vth_{11}(2z)$ change sign.
Coefficients containing a single $\vth_{11}$, however, change sign.
With the convention above
that $(\Phi(\xi)^{-1})^*$ induces translation by $\xi$, the
quarter-period formulas, in Mumford's notation, are as follows.  To
make them more readable, write $\vth_{ab}=\vth_{ab}(\eta)$.
\begin{equation}\label{eq.SS.Phi.one.fourth.Mumford}
\Phi\!\left(\tfrac14\right)(X_{11},X_{00},X_{01},X_{10})
=\left(X_{10},\frac{\vth_{10}\vth_{00}}{\vth_{11}\vth_{01}}X_{01},
-\frac{\vth_{10}\vth_{01}}{\vth_{11}\vth_{00}}X_{00},
\frac{\vth_{10}^2}{\vth_{11}^2}X_{11}\right).
\end{equation}
\begin{equation}\label{eq.SS.Phi.tau.fourth.Mumford}
\Phi\!\left(\tfrac14\tau\right)(X_{11},X_{00},X_{01},X_{10})
=\left(X_{01},-\ii\frac{\vth_{01}\vth_{00}}{\vth_{11}\vth_{10}}X_{10},
-\frac{\vth_{01}^2}{\vth_{11}^2}X_{11},
\ii\frac{\vth_{10}\vth_{01}}{\vth_{11}\vth_{00}}X_{00}\right).
\end{equation}
\begin{equation}\label{eq.SS.Phi.one.plus.tau.fourth.Mumford}
\Phi\!\left(\tfrac14+\tfrac14\tau\right)(X_{11},X_{00},X_{01},X_{10})
=\left(-\frac{\vth_{11}}{\vth_{00}}X_{00},
-\ii\frac{\vth_{00}}{\vth_{11}}X_{11},
-\frac{\vth_{01}}{\vth_{10}}X_{10},
\ii\frac{\vth_{10}}{\vth_{01}}X_{01}\right).
\end{equation}
The formulas for the three nonzero points of $E_\tau[2]$ contain no
theta constants and are unchanged:
\begin{align}
 \Phi\!\left(\tfrac12\right)
 (X_{11},X_{00},X_{01},X_{10})
   &=\phantom{(}(X_{11},-X_{00},-X_{01},X_{10}),\label{eq.SS.Phi.half}\\
 \Phi\!\left(\tfrac12\tau\right)
 (X_{11},X_{00},X_{01},X_{10})
   &=\phantom{(}(X_{11},-X_{00},X_{01},-X_{10}),\label{eq.SS.Phi.tau.half}\\
\Phi\!\left(\tfrac12+\tfrac12\tau\right)
 (X_{11},X_{00},X_{01},X_{10})
   &=\phantom{(}(X_{11},X_{00},-X_{01},-X_{10}).\label{eq.SS.Phi.one.plus.tau.half}
\end{align}
Direct substitution in the theta series shows that the inverse-dual
projective transformations defined by the six displayed formulas act
on $E_\tau$ as translation by, respectively,
\[
 \tfrac14,\quad \tfrac14\tau,\quad
 \tfrac14(1+\tau),\quad \tfrac12,\quad
 \tfrac12\tau,\quad \tfrac12(1+\tau).
\]
Consequently, Proposition~\ref{prop.SS.automorphisms} shows that each
displayed linear map preserves $R_\eta$.

\appendix
\section{A dictionary of theta-function notations}
\label{app:theta.dictionary}

The sign conventions for the odd Jacobi theta function vary between
standard sources.  The following table is organized by function: all
entries in a row denote the same analytic function, subject to the
argument convention stated below.  This makes
the exceptional minus sign in the first row explicit.

In Mumford's normalization, for $a,b\in\{0,1\}$,
\[
 \vartheta_{ab}(z\mid\tau)
 =\sum_{n\in\ZZ}
 \exp\!\left(
  \pi\ii\left(n+\frac a2\right)^2\tau
  +2\pi\ii\left(n+\frac a2\right)
            \left(z+\frac b2\right)
 \right).
\]

\begin{center}
\renewcommand{\arraystretch}{1.25}
\scriptsize
\setlength{\tabcolsep}{2.1pt}
\begin{tabular}{c|c|c|c|c|c|c|c}
This paper & Mumford & Weber & SS & KZ & Sklyanin & RR & Whittaker--Watson \\
\cite[p.~17]{MumfordTataI}
& \cite[p.~17]{MumfordTataI}
& \cite[vol.~III]{Weber}
& \cite{SmithStafford}
& \cite{KZ}
& \cite{Sklyanin1983}
& \cite{RR}
& \cite[Ch.~21]{WhittakerWatson} \\
\hline
$\vartheta_{11}$ & $\vartheta_{11}$ & $-\vartheta_{11}$ & $-\theta_{11}$ & $-\theta_1$ & $-\theta_1$ & $-\theta_1$ & $-\vartheta_1$ \\
$\vartheta_{00}$ & $\vartheta_{00}$ & $\phantom{-}\vartheta_{00}$ & $\phantom{-}\theta_{00}$ & $\phantom{-}\theta_3$ & $\phantom{-}\theta_3$ & $\phantom{-}\theta_3$ & $\phantom{-}\vartheta_3$ \\
$\vartheta_{01}$ & $\vartheta_{01}$ & $\phantom{-}\vartheta_{01}$ & $\phantom{-}\theta_{01}$ & $\phantom{-}\theta_4$ & $\phantom{-}\theta_4$ & $\phantom{-}\theta_4$ & $\phantom{-}\vartheta_4$ \\
$\vartheta_{10}$ & $\vartheta_{10}$ & $\phantom{-}\vartheta_{10}$ & $\phantom{-}\theta_{10}$ & $\phantom{-}\theta_2$ & $\phantom{-}\theta_2$ & $\phantom{-}\theta_2$ & $\phantom{-}\vartheta_2$
\end{tabular}
\end{center}

Equivalently, Weber's functions are numbered
$\vartheta_{11},\vartheta_{10},\vartheta_{00},\vartheta_{01}$ as
$\theta_1,\theta_2,\theta_3,\theta_4$, respectively; this is the
convention used by Kharchev--Zabrodin, Sklyanin, and
Rains--Ruijsenaars.  In the last column, Whittaker--Watson's argument
is $\pi z$ and their nome is $q=e^{\pi\mathrm{i}\tau}$; thus, for
example, its first entry abbreviates
$-\vartheta_1(\pi z,q)=\vartheta_{11}(z\mid\tau)$ in Mumford's
normalization.  Sklyanin uses no theta-function notation in
\cite{Sklyanin1982}.  The SS column records that Smith--Stafford write
$\theta_{ab}$ for Weber's $\vartheta_{ab}$.

Also, throughout this paper
$\theta(z)=\theta_1(z)=-\vartheta_{11}(z)$, with $\vartheta_{11}$
normalized as in Mumford.

\section{A dictionary of generator conventions}
\label{app:generator.dictionary}

We record here the conventions in this paper and others that we reference.
Throughout this appendix, we assume $\overline\eta\notin E_\tau[2]$.
Under the isomorphism $A(\eta)\cong S(\eta)$, the entries in each row of
the following table represent the same generator; in its first column,
$x_{ab}$ denotes the class in $A(\eta)$ of $X_{ab}\in V\subset TV$.

\begin{center}
\renewcommand{\arraystretch}{1.25}
\setlength{\tabcolsep}{3.5pt}
\begin{tabular}{c|c|c|c}
\S2.1 & \cite{SmithStafford,StaffordRegularity}
 & \cite{CSExotic} & \cite{Sklyanin1982,Sklyanin1983,RR} \\
\hline
$x_{11}$ & $x_0$ & $x_0$ & $S_0$ \\
$x_{00}$ & $x_1$ & $x_1$ & $-\ii S_2$ \\
$x_{01}$ & $x_2$ & $x_2$ & $\ii S_3$ \\
$x_{10}$ & $x_3$ & $x_3$ & $\ii S_1$
\end{tabular}
\end{center}

In this paper we also use the relabelling
$(x_{11},x_{00},x_{01},x_{10})=(x_0,x_1,x_2,x_3)$.

\section{A geometric dimension calculation}

For completeness, we give a second, vector-bundle proof of the
dimension statement in Theorem~\ref{prop:addition-formulas-relations}(3).

\begin{lemma}
\label{lem.dimension.R.eta}
For every $\eta\in\CC$, $\dim R_\eta=6$.
\end{lemma}

\begin{proof}
Let
$p_1,p_2:E_\tau\times E_\tau\to E_\tau$ be the two projections, write
$\xi$ for the coordinate on the second factor, and let
\[
 a:E_\tau\times E_\tau\longrightarrow E_\tau,
 \qquad a(z,\xi)=z+\xi.
\]
For $\mathcal L=\mathcal O_{E_\tau}(4[0])$, put
\[
 \mathcal H=(p_2)_*\bigl(p_1^*\mathcal L\otimes a^*\mathcal L\bigr).
\]
Its fiber at $\overline{\eta}$ is
$H^0(E_\tau,\mathcal L\otimes t_\eta^*\mathcal L)$.  Since
$\deg(\mathcal L\otimes t_\eta^*\mathcal L)=8$ for all $\eta$,
$H^1(E_\tau,\mathcal L\otimes t_\eta^*\mathcal L)=0$ and
$h^0(E_\tau,\mathcal L\otimes t_\eta^*\mathcal L)=8$ for all $\eta$.
Cohomology and base change therefore show that $\mathcal H$ is a vector
bundle of rank eight.

Every $r\in V^{\otimes2}$ is a section of
$p_1^*\mathcal L\otimes p_2^*\mathcal L$ on
$E_\tau\times E_\tau$.  For each $\eta\in\CC$, restrict $r$ to the
graph $\Gamma_\eta$ of $t_\eta$, parametrized by
$z\mapsto(z,z+\overline\eta)$.  The resulting section
$z\mapsto r(z,z+\overline\eta)$ belongs to
$H^0(E_\tau,\mathcal L\otimes t_\eta^*\mathcal L)$.  As
$\overline\eta$ varies in $E_\tau$, these restriction maps give a
morphism of vector bundles
\[
 V^{\otimes2}\otimes\mathcal O_{E_\tau}\longrightarrow\mathcal H.
\]
Proposition~\ref{prop.dim.Keta} says that it is surjective on every
fiber.  Its kernel
$\mathcal K$ is therefore a vector bundle of rank eight whose fiber at
$\overline{\eta}$ is $K_\eta$.

Let
\[
 \operatorname{ev}:V^{\otimes2}\longrightarrow\CC^4,
 \qquad r\longmapsto
 \bigl(r(\epsilon_0,\epsilon_0),\ldots,
 r(\epsilon_3,\epsilon_3)\bigr).
\]
The fiber of $\mathcal K$ at $\overline\eta$ is $K_\eta$, the space of
tensors that vanish on $\Gamma_\eta$.  By definition, $R_\eta$ is the
subspace of $K_\eta$ consisting of tensors that also vanish at the four
vertex pairs $(e_i,e_i)$, equivalently, at the four pairs
$(\epsilon_i,\epsilon_i)$.  The fiber of the kernel of
$\operatorname{ev}|_{\mathcal K}$ at $\overline{\eta}$ is therefore
$R_\eta$.
If $\overline{\eta}\notin E_\tau[2]$,
Theorem~\ref{prop:addition-formulas-relations} gives
$\dim R_\eta=6$, so $\operatorname{ev}|_{K_\eta}$ has rank two.  After
choosing local frames for $\mathcal K$ and the trivial bundle
$\CC^4\otimes\mathcal O_{E_\tau}$, the morphism
$\operatorname{ev}|_{\mathcal K}$ is represented by a $4\times8$
matrix of regular functions.  All its $3\times3$ minors vanish on the
dense open subset $E_\tau-E_\tau[2]$, and hence vanish everywhere.
Its rank is therefore at most two at every point.

On the other hand, $C_\eta$ is a two-dimensional subspace of $K_\eta$
for every $\eta$, and the proof of
Proposition~\ref{prop:intersection-R-C} shows that evaluation at the four
vertices is injective on $C_\eta$.  Hence
$\operatorname{ev}|_{K_\eta}$ has rank at least two.  Its rank is exactly
two for every $\eta$, whence
$\dim R_\eta=\dim K_\eta-2=6$.
\end{proof}

\section{A uniform basis for the Rains--Ruijsenaars relation space}
\label{app:uniform.RR.basis}

The proof of Lemma~\ref{prop:six.independent.RR.relations} uses
different configurations of secant lines according as
$\overline\eta$ does or does not have order four.  The following
result gives a single choice that works for every $\eta$.  Its proof,
which is essentially due to Rains--Ruijsenaars, uses the product
formula stated as (B.44) in \cite{RR}.  We include a
proof of that formula in order to make the argument self-contained.
We use the abbreviation
$\theta(x\mathbin\pm y)=\theta(x+y)\theta(x-y)$.

\begin{proposition}\label{prop:uniform.six.independent.RR.relations}
Let $\nu=\frac12\eta$.
Fix four different points
$\overline p_1,\ldots,\overline p_4\in E_\tau$ that are not coplanar
(i.e., their sum is $\ne0$), and fix preimages
$p_1,\ldots,p_4\in\CC$.  For $1\leq m<n\leq4$,
let $r,s$ be the two indices different from $m,n$, and put
\[
 G_{mn}:=F_\eta(-p_r,-p_s,-p_n,-p_m).
\]
\begin{enumerate}
\item The six functions $G_{mn}$ are linearly independent elements of
$R_\eta^{\mathrm{RR}}$.
\item Each $G_{mn}$ is of the form
$u_1\otimes v_2-u_2\otimes v_1$.
\item $\{v_1=v_2=0\}$ is the secant line
\[
 L_{mn}:=\overline{\overline{p_r-\nu},\overline{p_s-\nu}}.
\]
\item Consequently, when $\overline\eta\notin E_\tau[4]$, $G_{mn}$
is, up to a nonzero scalar multiple, ``the'' relation associated to
$L_{mn}$ by Levasseur--Smith.
\item The six lines $L_{mn}$ are the six edges of the configuration
formed by joining the images of
$p_1-\nu,\ldots,p_4-\nu$ in $E_\tau$.
\item When the images of those points in $E_\tau$ are not coplanar,
these lines are the edges of a tetrahedron.
\end{enumerate}
\end{proposition}

\begin{proof}
We first prove the product formula stated as (B.44) in \cite{RR}.
Clearly, $F_\eta=AD$, where
\begin{align}
A={}&\theta(z_1+a-\nu)\theta(z_1+b-\nu)
    \theta(z_2+a+\nu)\theta(z_2+b+\nu),\notag\\
D={}&
 \theta(z_1+c_1+\nu)\theta(z_1-a-b-c_1+\nu)
 \theta(z_2+c_2-\nu)\theta(z_2-a-b-c_2-\nu)\notag\\
&-\theta(z_1+c_2+\nu)\theta(z_1-a-b-c_2+\nu)
 \theta(z_2+c_1-\nu)\theta(z_2-a-b-c_1-\nu).
 \label{eq.RR.D.definition}
\end{align}
To simplify $D$, apply Weierstrass's three-term
identity \eqref{eq.Weierstrass.three-term} with
\[
 z=z_1+\nu-\tfrac12(a+b),\quad
 \alpha=c_2+\tfrac12(a+b),\quad
 \beta=c_1+\tfrac12(a+b),\quad
 \gamma=z_2-\nu-\tfrac12(a+b).
\]
Substitution in the three-term identity, followed by use of the
oddness of $\theta$, gives
\[
 D=\theta(z_1+z_2-a-b)\theta(z_1-z_2+\eta)
   \theta(c_1+c_2+a+b)\theta(c_1-c_2).
\]
Multiplying this equality by $A$ gives
\begin{equation}\label{eq.RR.B44.product}
\begin{aligned}
F_\eta(a,b,c_1,c_2;z_1,z_2)
={}&\theta(c_1+c_2+a+b)\theta(c_1-c_2)\\
&\cdot\theta(z_1+a-\nu)\theta(z_1+b-\nu)
       \theta(z_2+a+\nu)\theta(z_2+b+\nu)\\
&\cdot\theta(z_1+z_2-a-b)\theta(z_1-z_2+\eta).
\end{aligned}
\end{equation}

For $m=1,2,3,4$, define
\[
 \phi_m(z):=\theta\!\bigl(z+\sum_{j\ne m}p_j\bigr)
             \prod_{j\ne m}\theta(z-p_j).
\]
These functions belong to $V$.  Moreover, $\phi_m(p_n)=0$ if
$m\ne n$, whereas
\[
 \phi_m(p_m)=
 \theta(p_1+p_2+p_3+p_4)
 \prod_{j\ne m}\theta(p_m-p_j)\ne0.
\]
It follows that $\{\phi_1,\phi_2,\phi_3,\phi_4\}$ is a basis for $V$.

Set
\[
 \kappa_{mn}:=
 \theta(-p_1-p_2-p_3-p_4)
 \theta(p_m-p_n).
\]
Our assumptions imply that $\kappa_{mn}\ne0$.  At $\eta=0$, the
definition of $F_\eta$ gives
\begin{equation}\label{eq.RR.exterior.basis}
 G_{mn}=\phi_m\otimes\phi_n-\phi_n\otimes\phi_m.
\end{equation}
For the complementary indices $r,s$, put
\[
 Q_{mn}(y_1,y_2):=
 \theta(y_1-p_r)\theta(y_1-p_s)
 \theta(y_2-p_r)\theta(y_2-p_s)
 \theta(y_1+y_2+p_r+p_s).
\]
After the invertible change of variables
$y_1=z_1-\nu$ and $y_2=z_2+\nu$,
\eqref{eq.RR.B44.product} becomes
\[
 G_{mn}(z_1,z_2)=
 \kappa_{mn}Q_{mn}(y_1,y_2)\theta(y_1-y_2+2\eta).
\]
When $\eta=0$, the six functions
$\kappa_{mn}Q_{mn}(y_1,y_2)\theta(y_1-y_2)$ are the six exterior
products in \eqref{eq.RR.exterior.basis}, and are therefore linearly
independent.  It follows that the six $Q_{mn}$ are linearly
independent: a linear relation among them, multiplied by the nonzero
function $\theta(y_1-y_2)$, would give a linear relation among those
six exterior products.

Now fix an arbitrary $\eta$.  Translation from $(z_1,z_2)$ to
$(y_1,y_2)$ is an invertible linear operation on spaces of functions.
If a linear combination of the six $G_{mn}$ were zero, the last
display would give
\[
 \theta(y_1-y_2+2\eta)
 \sum_{m<n}\lambda_{mn}\kappa_{mn}Q_{mn}(y_1,y_2)=0.
\]
The first factor is not the zero function, so the sum must be zero.
The independence of the $Q_{mn}$ and the fact that every
$\kappa_{mn}$ is nonzero imply that all $\lambda_{mn}$ are zero.
This proves (1).  Parts (2) and (3) follow from
Lemma~\ref{lem:RR.functions.vanish.on.graph}, applied with
$(a,b,c_1,c_2)=(-p_r,-p_s,-p_n,-p_m)$.  Part (4) follows from
Corollary~\ref{prop:RR.relations.and.LS.secants}, and (5) and (6) follow
from the definition of the lines $L_{mn}$.
\end{proof}

\end{document}